\documentclass[11pt]{amsart}
\usepackage[margin=1.1in]{geometry}
\usepackage{amsmath,amssymb,amsthm,mathtools,microtype}
\usepackage[hidelinks]{hyperref}
\usepackage{xcolor}

\newtheorem{theorem}{Theorem}[section]
\newtheorem{lemma}[theorem]{Lemma}
\newtheorem{proposition}[theorem]{Proposition}
\newtheorem{corollary}[theorem]{Corollary}
\theoremstyle{remark}
\newcommand{\C}{\mathbb C}\newcommand{\N}{\mathbb N}
\newcommand{\Hs}{\mathcal H^s}
\newcommand{\dist}{\operatorname{dist}}\newcommand{\diam}{\operatorname{diam}}
\newcommand{\Desc}{\operatorname{Desc}}
\DeclareMathOperator{\gen}{gen}

\title[Dimension drop]{Failure of dimension drop for harmonic measure on Ahlfors regular boundaries of dimension close to $1$}

\author{Yingying Cai}
\address{Yingying Cai: Departament de Matemàtiques, Universitat Autònoma de Barcelona, Barcelona, Catalonia.}
\email{yingying.cai@uab.cat}

\author{Cole Jeznach}
\address{Cole Jeznach: Departament de Matemàtiques, Universitat Autònoma de Barcelona, Barcelona, Catalonia.}
\email{cole.jeznach@uab.cat}

\author{Xavier Tolsa}
\address{Xavier Tolsa: ICREA, Barcelona, Departament de Matemàtiques, Universitat Autònoma de Barcelona, and Centre de Recerca Matemàtica, Barcelona, Catalonia.}
\email{xavier.tolsa@uab.cat}

\date{\today}

\thanks{The authors were partially supported by MICIU (Spain) under the grant PID2024-160507NB-I00. C.J. was supported by NSF-DMS-2503326.
}

\begin{document}

\begin{abstract}
For every $0<s<1$ we construct a compact $s$-Ahlfors regular set
$K\subset\mathbb C$ for which harmonic measure of $\mathbb C\setminus K$
with pole at infinity is comparable to $\mathcal H^s|_K$. In particular,
there is no dimension drop for any dimension $s<1$. The construction is
$N$-adic and based on configurations close to the $N$-th roots of unity.
We adjust the positions of the children by small polynomial deformations
so that the logarithmic potential of the natural Bernoulli measure
becomes constant in the limit. A first-mode correction gives the range
$s<2/3$, while corrections of arbitrary finite order yield the full range
$0<s<1$.
\end{abstract}

\maketitle

\setcounter{tocdepth}{1}
\tableofcontents

\section{Introduction}

Let $\Omega\subset\mathbb R^n$ be a connected domain with compact
boundary. For $x\in\Omega$, harmonic measure $\omega_\Omega^x$ is the
measure representing solutions of the Dirichlet problem for the Laplace
equation. More precisely, if $\Omega$ is Wiener regular and
$f\in C(\partial\Omega)$, let $u$ be the harmonic function in $\Omega$
with continuous boundary values $f$ (and, for unbounded domains in
dimensions $n\ge3$, with $u$ tending to zero at infinity). Then
\[
 u(x)=\int_{\partial\Omega}f(y)\,d\omega_\Omega^x(y).
\]
In the planar situation considered in this paper, when $\Omega$ is the
complement of a nonpolar compact set, we also use harmonic measure with
pole at infinity, denoted by $\omega_\Omega^\infty$. It coincides with
the logarithmic equilibrium measure of $\partial\Omega$ when the latter
is regular. In particular, this characterization will be useful for our
construction.

The study of harmonic measure and its relation to the geometry of the
boundary is a classical topic in potential theory and geometric measure
theory. For a finite Borel measure $\nu$, recall that its Hausdorff
dimension is defined by
\[
 \dim_{\mathcal H}\nu
 =\inf\{\dim_{\mathcal H}F:\ F\text{ Borel and }\nu(F^c)=0\}.
\]
The strict inequality
\[
 \dim_{\mathcal H}\omega_\Omega^x<\dim_{\mathcal H}\partial\Omega
\]
is called \emph{dimension drop}. In the plane, Jones and
Wolff~\cite{JW88:dim} proved that $\dim_{\mathcal H}\omega_\Omega^x\le1$
for arbitrary domains. 
In particular, dimension drop always occurs when the Hausdorff dimension of the boundary is strictly larger than \(1\).
In higher dimensions, Bourgain~\cite{Bou87}
showed that for each $n\ge3$ there exists $\varepsilon_n>0$ such that
$\dim_{\mathcal H}\omega_\Omega^x\le n-\varepsilon_n$ for every domain
$\Omega\subset\mathbb R^n$. On the other hand, Wolff~\cite{Wol95}
constructed a domain in $\mathbb R^3$ whose harmonic measure has
dimension strictly greater than $2$. Determining the optimal upper
bound in higher dimensions remains an important open problem.

Dimension drop also occurs for many self-similar and self-conformal
fractal boundaries. Early results in this direction include the work
of Carleson, Makarov--Volberg, and Volberg; see also Batakis~\cite{Bat96},
who proved dimension drop for a class of planar Cantor sets that
includes self-similar examples, and the later work~\cite{BZ15:cantor}.

A natural setting for the study of dimension drop is that of
$s$-Ahlfors regular boundaries: a compact set $K$ is $s$-Ahlfors
regular if there exists $C\ge1$ such that
\[
 C^{-1}r^s\leq\mathcal H^s(K\cap B(x,r))\leq Cr^s,
 \qquad x\in K,\quad 0<r<\operatorname{diam}K.
\]
For $s>n-1$, Azzam~\cite{Azzam20:dim} proved that an $s$-Ahlfors
regular boundary in $\mathbb R^n$ necessarily exhibits dimension drop.
For $s<n-1$, but sufficiently close to $n-1$, he also obtained
dimension drop under an additional uniform non-flatness condition.
A recent result of Cai and Tolsa~\cite{CT26} quantifies this phenomenon:
in particular, an appropriate uniform lower bound $\beta_2\ge\delta_0$
for the $L^2$ beta coefficients yields dimension drop in the range
$s>n-1-c\delta_0^2$ (with the corresponding precise assumptions),
and an analogous quantitative result holds in the plane under Azzam's
bilateral condition involving $\beta_\infty$ and the hole coefficient.
See also the work of Pathak \cite{Pathak} for a less precise quantification.

For boundaries of dimension below the codimension-one threshold, it
was natural to ask whether Ahlfors regularity might also force dimension
drop. In this direction, Volberg~\cite{Vol22} formulated a related,
stronger conjecture (Conjecture~1.9 in that reference), asserting,
roughly speaking, that under a suitable accessibility hypothesis the
Green-function estimate
$G(z)\approx\operatorname{dist}(z,\partial\Omega)^s$ should not hold
for $s<1$ in the planar case. David, Jeznach and Julia~\cite{DJJ23}
disproved this conjecture by constructing planar $s$-Ahlfors regular
Cantor sets of small fractional dimension whose harmonic measure is
comparable to $\mathcal H^s|_K$; see also their discussion of Volberg's
precise formulation. Their construction yields examples on a line for
$0<s<0.2497$, and examples not constrained to a line for dimensions
slightly larger than $0.406$. In particular, these examples show that
Ahlfors regularity alone does not force dimension drop below dimension
one. The present paper extends the range of such planar examples to
every $0<s<1$.

When $K$ is contained in a line, the threshold $s=1/2$ is sharp: dimension drop holds for $s$-Ahlfors regular compact subsets of a line when $1/2\le s<1$~\cite{T24}, whereas a related $N$-adic construction yields examples without dimension drop for every $0<s<1/2$ (see the companion paper~\cite{CJT26}). This makes it particularly interesting to investigate what happens when $K$ is not required to lie on a line.

Our main result is that in the non-flat planar setting there is no
analogous obstruction below dimension one.

\begin{theorem}\label{thm:main}
For every $0<s<1$ there exists a compact $s$-Ahlfors regular set
$K\subset\mathbb C$ such that, with $\Omega=\mathbb C\setminus K$,
\[
 \omega^\infty\approx\mathcal H^s|_K.
\]
Here $\omega^\infty$ denotes harmonic measure for $\Omega$ with pole at
infinity, and comparability means that for some $C\geq1$,
\[
 C^{-1}\mathcal H^s(F)\leq\omega^\infty(F)
 \leq C\mathcal H^s(F),\qquad F\subset K\ \text{Borel}.
\]
In particular, $\dim_{\mathcal H}\omega^\infty=s$.
\end{theorem}

\subsection*{Idea of the construction}
Following the potential-theoretic approach of David--Jeznach--Julia,
we construct a nested $N$-adic family: each parent point is replaced by
$N$ children placed close to a circle. For a prescribed $s<1$ we take
$\lambda=N^{-1/s}$, so that $N\lambda^s=1$. The reference positions are
the roots of unity $\zeta_j=e^{2\pi i j/N}$. Unlike the Chebyshev nodes
used for a construction contained in a line in \cite{CJT26}, these reference points have
an exactly constant discrete logarithmic interaction:
\[
 \sum_{k\neq j}\log|\zeta_j-\zeta_k|=\log N
 \qquad(0\leq j<N).
\]
However, this symmetry alone does not persist at successive scales:
the field generated by other families of points varies among the children
of a given parent.

To compensate for this variation, we deform each circular configuration
by a polynomial
\[
 V_Q(\zeta)=\zeta+\sum_{r=1}^M d_{Q,r}\zeta^{r+1},
 \qquad |d_{Q,r}|\lesssim_M\tau^r,\quad \tau=N\lambda,
\]
and introduce a scalar dilation to correct the mean potential. The
polynomial coefficients successively cancel the first $M$ nonconstant
Fourier modes of the local potential. For fixed $M$, the remaining error
is of size $O_M(N\tau^{M+1})$. Since
\[
 N\tau^{M+1}=N^{M+2-(M+1)/s}\longrightarrow0
 \quad\text{if}\quad s<\frac{M+1}{M+2},
\]
choosing $M$ sufficiently large yields any prescribed $s<1$.

The limiting measure is the pushforward of the uniform Bernoulli
measure on the symbolic Cantor space. Separation of cylinders implies
its $s$-Ahlfors regularity. On the other hand, the corrections make its
logarithmic potential constant on the limiting set. Hence it coincides
with the equilibrium measure of $K$, that is, harmonic measure with pole
at infinity.

\subsection*{Organization of the paper}
We begin with the circular reference configuration and explain in detail
the correction of the first Fourier mode; this already gives the range
$s<2/3$. We then develop corrections of arbitrary polynomial order,
prove the corresponding inductive estimates, and pass to the limiting
equilibrium measure. The proof of Theorem~\ref{thm:main} is completed at
the end of the paper.

\subsection*{AI Disclosure} 
 The authors used the large language model GPT-5.6 to explore mathematical ideas and develop proofs in the preparation of this paper. In particular, the initial idea of considering circular Cantor constructions with a large number $N$ of children per parent, and then letting $N\to\infty$ to obtain examples without dimension drop for the full range $0<s<1$, originated with the authors. The final result was developed through extensive iterative interactions between the authors and the LLM. The authors also introduced significant simplifications and improvements to some of the arguments initially proposed by the LLM. The authors have carefully checked all the proofs and take full responsibility for the mathematical content of the paper.

\section{The circular construction and its first-order correction}
\subsection{Parameters and construction}
\label{subsec:first-order-parameters}
In this section and through Section~\ref{sec:first-mode-induction} we first work in the range $0<s<2/3$. The arbitrary-order extension will be treated separately.
Fix $0<s<2/3$. Let $N\ge4$ be large and put
\begin{equation}\label{eq:lambda} \lambda=N^{-1/s},\qquad N\lambda^s=1.\end{equation}
Set
\begin{equation}\label{eq:tau} \tau=N\lambda=N^{1-1/s}.\end{equation}
For $0<s<1$, one has $0<\tau<1$.
Fix once and for all $\rho=1/10$ and a sufficiently small absolute
constant $0<\varepsilon_0\ll1$. The quantitative hypothesis used below is
\begin{equation}\label{eq:smallness} N\tau^2=N^3\lambda^2\le\varepsilon_0.\end{equation}
This can be achieved by taking $N$ large precisely when $s<2/3$.

Let $\zeta_j=e^{2\pi i j/N}$, $0\le j<N$. We use
\begin{equation}\label{eq:roots} \sum_{j=0}^{N-1}\zeta_j^k=0\ (1\le k<N),\qquad \prod_{k\ne j}|\zeta_j-\zeta_k|=N.\end{equation}
The geometric estimates require only
that these constants be small enough. Throughout the first-order
construction the admissible parameters satisfy
\begin{equation}\label{eq:admissible}
 |\log b_Q|\le\varepsilon_0,\qquad |d_Q|\le C_0\tau,
\end{equation}
where $C_0$ is a fixed constant. We take $N$ large enough that $C_0\tau$
is small. The parameters will be chosen inductively in
Section~\ref{sec:first-mode-induction}; for now we describe the set
and measure associated with any family of admissible parameters.

Let $\mathcal A=\{0,\ldots,N-1\}$. 
For a finite word $Q$, we denote by $\gen(Q)$
its generation, that is, the number of letters in $Q$.  
For every vertex $Q$ introduce
\begin{equation}\label{eq:VQ} V_Q(\zeta)=\zeta+d_Q\zeta^2,\end{equation}
and a positive scalar $b_Q$. The parameter $d_Q$ corrects the first Fourier mode and $b_Q$ corrects the mean potential. By \eqref{eq:roots}, $\sum_jV_Q(\zeta_j)=0$.

To the word $Q$ we associate $x_Q\in\C$ and a scale
\begin{equation}\label{eq:rQ} r_Q=\lambda^{\gen(Q)}.
\end{equation}
Starting from $x_\varnothing=0$, define
\begin{equation}\label{eq:children}
x_{Qj}=x_Q+\rho b_Qr_Q V_Q(\zeta_j),\qquad r_{Qj}=\lambda r_Q.
\end{equation}
The factor $b_Q$ changes only the position of the children, not their intrinsic scales. In particular, no product of the $b_R$ enters $r_Q$.

Set $K_n=\{x_Q:\gen(Q)=n\}$. We will often identify
$Q$ with the corresponding point $x_Q\in\C$, by a slight
abuse of notation. Thus, for example, we may write
$Q\in K_n$ instead of $x_Q\in K_n$. If $m\ge\gen(Q)$, let $\Desc_m(Q)$ be the descendants of $Q$ in $K_m$. The barycenter identity gives
\begin{equation}\label{eq:barycenter}
x_Q=\frac1N\sum_jx_{Qj}=\frac1{N^{m-\gen(Q)}}\sum_{x_R\in\Desc_m(Q)}x_R.
\end{equation}

\subsubsection*{The limiting compact set and its natural measure}
Let $\Sigma=\mathcal A^{\N}$ be the space of infinite words. If
$\omega\in\Sigma$, write $\omega|n$ for its prefix of length $n$.
The displacement bound in \eqref{eq:children} gives
\[
 |x_{Qj}-x_Q|\le C r_Q.
\]
Consequently, for $m>n$,
\[
 |x_{\omega|m}-x_{\omega|n}|
 \le C\sum_{k=n}^{m-1}\lambda^k
 \le \frac{C}{1-\lambda}\lambda^n.
\]
Thus the centers along every infinite branch converge, uniformly in
$\omega$. Define
\begin{equation}\label{eq:coding-K}
 \pi(\omega)=\lim_{n\to\infty}x_{\omega|n},
 \qquad K=\pi(\Sigma).
\end{equation}
Equip $\Sigma$ with the ultrametric
\[
 d_\Sigma(\omega,\omega')=\lambda^{L(\omega,\omega')}
 \quad(\omega\ne\omega'),\qquad d_\Sigma(\omega,\omega)=0,
\]
where $L(\omega,\omega')$ is the length of their maximal common
prefix. The preceding bound implies that $\pi$ is continuous.
Since $\Sigma$ is compact, $K$ is compact as well.

For a finite word $Q$, let
\[
 [Q]=\{\omega\in\Sigma:\omega|_{\gen(Q)}=Q\},
 \qquad K_Q=\pi([Q]).
\]
We call $K_Q$ the cylinder associated with $Q$. In particular,
\[
 K_\varnothing=K,\qquad K_Q=\bigcup_{j=0}^{N-1}K_{Qj},
 \qquad K_{Qj}\subset K_Q.
\]
The finite sets $K_n$ consist of the generation-$n$ centers: they
are approximations to $K$, not a decreasing family of compact sets.
Indeed, the branch estimate yields
\[
 d_H(K_n,K)\le C\lambda^n,
\]
where $d_H$ denotes Hausdorff distance. The nested objects in the
construction are the cylinders $K_Q$.

Let $\nu$ be the uniform Bernoulli probability measure on $\Sigma$,
so that $\nu([Q])=N^{-\gen(Q)}$, and define
\begin{equation}\label{eq:Bernoulli-mu}
 \mu=\pi_\#\nu.
\end{equation}
The separation estimates in Section~\ref{sec:geometry} will show
that $\pi$ is injective. In particular, the cylinders of any fixed
generation are disjoint and
\[
 \mu(K_Q)=N^{-\gen(Q)}=r_Q^s.
\]
They will also show that $\mu$ is $s$-Ahlfors regular.
At this point no equilibrium property is asserted: the purpose of
the later inductive choice of $b_Q,d_Q$ is to make the logarithmic
potential of this natural measure constant on $K$.

\subsection{The unperturbed circular configuration}
\label{subsec:unperturbed-circle}

To motivate the deformation in \eqref{eq:VQ}, we examine the basic
unperturbed circular configuration and record the identities behind
the first-mode correction.

Let
\[
    \zeta_j=e^{2\pi i j/N},
    \qquad j=0,\ldots,N-1.
\]
Consider first $N$ point masses of equal mass $1/N$ located at the
points $\zeta_j$. The discrete logarithmic potential generated by the
other $N-1$ points at $\zeta_j$ is
\[
    \Phi_N(\zeta_j)
    =
    \frac1N
    \sum_{\substack{k=0\\k\neq j}}^{N-1}
    \log\frac1{|\zeta_j-\zeta_k|}.
\]
By \eqref{eq:roots},
\[
    \prod_{\substack{k=0\\k\neq j}}^{N-1}
    |\zeta_j-\zeta_k|=N,
\]
and hence
\begin{equation}\label{eq:unperturbed-potential-constant}
    \Phi_N(\zeta_j)
    =
    -\frac{\log N}{N},
    \qquad j=0,\ldots,N-1.
\end{equation}
Thus the discrete potential is exactly constant on the reference
configuration.

This is the circular analogue of the role played by the
Gauss--Chebyshev nodes in the work \cite{CJT26}. There, the
nodes discretize the equilibrium measure of an interval. Here, the
equilibrium measure of a circle is normalized arclength measure, and
the equally spaced points provide its natural discrete counterpart.
In fact, in the circular case the constancy in
\eqref{eq:unperturbed-potential-constant} is exact.

We next examine what happens when the configuration is observed one
scale below. Fix the point $\zeta_0=1$ and consider a point of the
form
\[
    1+w,
    \qquad |w|<c/N.
\]
The unnormalized logarithmic potential produced by the other $N-1$
nodes is
\[
    H_N(w)
    =
    \sum_{k=1}^{N-1}
    \log\frac1{|1+w-\zeta_k|}.
\]
Notice that $H_N$ is not normalized by a factor $1/N$. This
normalization will be convenient later, since the quantities
$\Delta_1,\Delta_2,\Delta_3$ appearing in the inductive argument will
also be defined as unnormalized sums.

Using
\[
    \prod_{k=1}^{N-1}(1+w-\zeta_k)
    =
    \frac{(1+w)^N-1}{w},
\]
and
\[
    \prod_{k=1}^{N-1}(1-\zeta_k)=N,
\]
we obtain the exact identity
\begin{equation}\label{eq:HN-exact}
    H_N(w)-H_N(0)
    =
    -\Re\log
    \left(
        \frac{(1+w)^N-1}{Nw}
    \right).
\end{equation}

For $N|w|$ sufficiently small,
\[
    \frac{(1+w)^N-1}{Nw}
    =
    1+\frac{N-1}{2}w
    +
    \frac{(N-1)(N-2)}6w^2
    +O(N^3|w|^3).
\]
Taking the logarithm gives
\begin{equation}\label{eq:HN-expansion}
    H_N(w)-H_N(0)
    =
    -\frac{N-1}{2}\Re w-
    \frac{(N-1)(N-5)}{24}\Re(w^2)
    +O(N^3|w|^3).
\end{equation}
In particular, if
\[
    w=\lambda e^{i\theta},
\]
then, recalling that $\tau=N\lambda$,
\begin{equation}\label{eq:HN-angular-expansion}
    H_N(\lambda e^{i\theta})-H_N(0)
    =\,
    -\frac{N-1}{2}\lambda\cos\theta-
    \frac{(N-1)(N-5)}{24}
    \lambda^2\cos(2\theta)
    +O(\tau^3).
\end{equation}
Thus the leading nonconstant term is a first Fourier mode. In the
unnormalized potential $H_N$ its size is comparable to
$N\lambda=\tau$.

Formula \eqref{eq:HN-angular-expansion} explains why the perfectly
symmetric configuration cannot simply be iterated without a
correction. Although the potential is exactly constant on each
reference circle, the points belonging to a smaller circular
configuration centered at one of the nodes see the other nodes of the
previous generation asymmetrically. The dominant asymmetry is the
first Fourier mode in \eqref{eq:HN-angular-expansion}: points on the
inner and outer sides of the smaller circle do not see the preceding
generation in exactly the same way.

This motivates the introduction of a small deformation of the
reference configuration which cancels the first Fourier mode. The
simplest convenient form of such a deformation is the holomorphic
perturbation
\[
    V_d(\zeta)=\zeta+d\zeta^2,
\]
for some $d\in\mathbb C$.
We next record the linear response of the local discrete potential to
this perturbation.

Define
\[
    \Psi_{N,d}(j)
    =
    \sum_{\substack{k=0\\k\neq j}}^{N-1}
    \log|V_d(\zeta_j)-V_d(\zeta_k)|.
\]
Since
\[
    V_d(\zeta_j)-V_d(\zeta_k)
    =
    (\zeta_j-\zeta_k)
    \bigl(1+d(\zeta_j+\zeta_k)\bigr),
\]
we have, for $|d|$ sufficiently small,
\[
\begin{split}
    \Psi_{N,d}(j)-\Psi_{N,0}(j)
    &=
    \Re\sum_{k\neq j}
    \log\bigl(1+d(\zeta_j+\zeta_k)\bigr)=
    \Re\left[
        d\sum_{k\neq j}(\zeta_j+\zeta_k)
    \right]
    +O(N|d|^2).
\end{split}
\]
Using
\[
    \sum_{k\neq j}\zeta_k=-\zeta_j,
\]
we obtain
\[
    \sum_{k\neq j}(\zeta_j+\zeta_k)
    =
    (N-2)\zeta_j.
\]
Consequently,
\begin{equation}\label{eq:linear-response-unperturbed}
    \Psi_{N,d}(j)-\Psi_{N,0}(j)
    =
    (N-2)\Re(d\zeta_j)
    +O(N|d|^2).
\end{equation}

Thus the perturbation $d\zeta^2$ produces, to first order, precisely
a first Fourier mode. Moreover, the linear response has size
comparable to $N|d|$. This is the mechanism that will allow us to
cancel the leading first Fourier mode created by the preceding
generations.

We emphasize that the preceding computation is only intended to
motivate the choice of perturbation. In the inductive construction,
the first Fourier mode receives contributions not only from the
immediately preceding generation but also from all earlier
generations. Moreover, the normalization of those contributions is
most naturally expressed in terms of the quantities
$\Delta_1,\Delta_2,\Delta_3$ introduced below. We will therefore not
choose $d$ at this stage. Instead, after all these contributions have
been taken into account, we will prove that one can choose the
parameter $d_Q\in \mathbb C$ at each vertex $Q$ so that
\[
    |d_Q|\lesssim\tau
\]
and the complete first Fourier mode vanishes.

After this cancellation, the remaining nonconstant part of the
potential will be of size $O(N\tau^2)$. This is the origin of the
smallness condition
\[
    N\tau^2\ll1
\]
in the first-order construction considered in this section.

\section{Geometry and Ahlfors regularity}
\label{sec:geometry}

Throughout this section we assume the admissibility conditions
\eqref{eq:admissible}, with $\tau$ sufficiently small. The constants
in the localization and separation estimates are independent of $N$;
the Ahlfors regularity constants may depend on the fixed parameters
$N$ and $\lambda$.

\begin{lemma}\label{lem:V-bilip}
If $\tau$ is small enough, then
\[\frac12|\zeta-\eta|\le|V_Q(\zeta)-V_Q(\eta)|\le2|\zeta-\eta|\quad(\zeta,\eta\in\mathbb S^1).\]
\end{lemma}
\begin{proof}
$V_Q(\zeta)-V_Q(\eta)=(\zeta-\eta)(1+d_Q(\zeta+\eta))$.
\end{proof}

Consequently $|x_{Qj}-x_{Qk}|\approx r_Q|\zeta_j-\zeta_k|$ and, for $j\ne k$,
\begin{equation}\label{eq:sibling-min}|x_{Qj}-x_{Qk}|\ge c r_Q/N.\end{equation}

\begin{lemma}\label{lem:cylinder}
For every $Q$ and $m\ge\gen(Q)$, $\Desc_m(Q)\subset B(x_Q,Cr_Q)$. Hence the limiting cylinder $K_Q$ satisfies $\diam K_Q\le Cr_Q$.
\end{lemma}
\begin{proof}
The successive displacements are bounded by $Cr_Q,C\lambda r_Q,C\lambda^2r_Q,\ldots$ and form a geometric series.
\end{proof}

\begin{lemma}\label{lem:separation}
If $\tau$ is small enough, then
\[\dist(K_{Qj},K_{Qk})\ge c r_Q/N\qquad(j\ne k).\]
\end{lemma}
\begin{proof}
Combine \eqref{eq:sibling-min} with $\diam K_{Qj}+\diam K_{Qk}\le C\lambda r_Q=C\tau r_Q/N$.
\end{proof}

Recall that $\Sigma$, $d_\Sigma$, $\pi$, $K$, $\nu$, and $\mu$ were defined in Section~\ref{subsec:first-order-parameters}.

\begin{lemma}\label{lem:symbolic}
There are $c,C>0$ such that
\[\frac cN d_\Sigma(\omega,\omega')\le|\pi(\omega)-\pi(\omega')|\le C d_\Sigma(\omega,\omega').\]
\end{lemma}
\begin{proof}
If $Q$ is the maximal common prefix, the upper estimate follows from Lemma~\ref{lem:cylinder}, and the lower estimate from Lemma~\ref{lem:separation} applied to the two distinct children of $Q$.
\end{proof}

\begin{lemma}\label{lem:AD}
The measure $\mu$ is $s$-Ahlfors regular and $\mu\approx\Hs|_K$.
\end{lemma}
\begin{proof}
If $\lambda^{n+1}<r\le\lambda^n$, then
\[[\omega|_{n+1}]\subset B_\Sigma(\omega,r)\subset[\omega|_n],\]
so
\[N^{-(n+1)}\le\nu(B_\Sigma(\omega,r))\le N^{-n}.\]
Since $N^{-1}=\lambda^s$, $\nu(B_\Sigma(\omega,r))\approx r^s$. Lemma~\ref{lem:symbolic} transfers this estimate to $K$: if $x=\pi(\omega)$ and $c_0d_\Sigma\le|\pi(\omega)-\pi(\omega')|\le C_0d_\Sigma$, then
\[
\nu\!\left(B_\Sigma(\omega,r/C_0)\right)\le\mu(B(x,r))\le\nu\!\left(B_\Sigma(\omega,r/c_0)\right).
\]
Thus $\mu(B(x,r))\approx r^s$ for all sufficiently small $r$, and the remaining bounded range of scales is absorbed by changing the constant. The standard mass-distribution and $5r$-covering arguments then give $\mu\approx\Hs|_K$.
\end{proof}

\section{Discrete potentials and the three error terms}
Define
\begin{equation}\label{eq:gn}
g_n(x)=N^{-n}\sum_{\substack{y\in K_n\\y\ne x}}\log|x-y|,\qquad x\in K_n.
\end{equation}
If $x=x_{Qj}\in K_n$ and $Q=\widehat x\in K_{n-1}$ is its parent, then a direct regrouping gives
\begin{equation}\label{eq:decomposition}
g_n(x)=g_{n-1}(Q)+N^{-n}\bigl(\Delta_1(x)+\Delta_2(x)+\Delta_3(x)\bigr),
\end{equation}
where
\begin{align}
\Delta_1(x)&=\sum_{\substack{z\in\{x_{Qk}:0\le k<N\}\\z\ne x}}\log|z-x|,\label{eq:D1}\\
\Delta_2(x)&=N\sum_{\substack{y\in K_{n-1}\\y\ne Q}}\bigl(\log|y-x|-\log|y-Q|\bigr),\label{eq:D2}\\
\Delta_3(x)&=\sum_{\substack{y\in K_{n-1}\\y\ne Q}}\sum_{z\in\{x_{yk}:0\le k<N\}}\bigl(\log|z-x|-\log|y-x|\bigr).\label{eq:D3}
\end{align}
For a fixed $x\in K_n$, the term $\Delta_1$ is the local interaction between siblings, $\Delta_2$ measures the displacement from the parent in the exterior field, and $\Delta_3$ measures the error made by replacing each exterior family by its barycenter.



We recall some notation for the discrete Fourier
decomposition on the $N$-th roots of unity.
Let $\zeta_j=e^{2\pi i j/N}$, $0\le j<N$.
For a real-valued function $f$ on these points, write
\[
\widehat f(k)
=\frac1N\sum_{j=0}^{N-1}f(\zeta_j)\zeta_j^{-k},
\qquad
\overline f=\widehat f(0).
\]
For $1\le k<N/2$, define
\[
P_k f(\zeta_j)
=2\operatorname{Re}\bigl(\widehat f(k)\zeta_j^k\bigr).
\]
For $2\le m\le N/2$, set
\[
P_{\ge m}f
=f-\overline f-\sum_{k=1}^{m-1}P_k f.
\]
In particular,
\[
f=\overline f+P_1f+P_{\ge2}f.
\]
We also write $P_{\mathrm{nc}}f=f-\overline f$.

\section{The first-mode inductive construction}
\label{sec:first-mode-induction}

\subsection{The inductive goal and the three terms}

We explain how to pass from one generation of the circular Cantor
construction to the next. Fix $n\ge2$ and suppose that $K_{n-1}$ is
already constructed. Every parameter $b_R,d_R$ used at a vertex
$R\in K_0\cup\cdots\cup K_{n-2}$ is now fixed.

At the current parents $Q\in K_{n-1}$, we have yet to choose
\[
 b_Q=e^{t_Q}>0,\qquad d_Q\in\mathbb C,
\]
and we write
\[
 \mathbf t=(t_Q)_{Q\in K_{n-1}},\qquad
 \mathbf d=(d_Q)_{Q\in K_{n-1}}.
\]
These parameters determine
\[
 x_{Qj}=Q+\rho b_Qr_Q(\zeta_j+d_Q\zeta_j^2),
 \qquad 0\le j<N.
\]

In the estimates below we first consider arbitrary admissible
trial parameters satisfying
\[
 |t_Q|\le\varepsilon_0,\qquad |d_Q|\le C_0\tau.
\]
Only after establishing the estimates do we choose the parameters
explicitly and verify these bounds.

Recall the exact decomposition
\[
 g_n(x_{Qj})
 =g_{n-1}(Q)+N^{-n}F_Q(\mathbf d,\mathbf t;j),
\]
where
\[
 F_Q(\mathbf d,\mathbf t;j)
 :=\Delta_1(x_{Qj})+\Delta_2(x_{Qj})+\Delta_3(x_{Qj}),
\]
and
\[
 \overline F_Q
 :=\frac1N\sum_{j=0}^{N-1}F_Q(\mathbf d,\mathbf t;j).
\]

Here the dependence on the full parameter vectors matters:
$\Delta_1$ and $\Delta_2$ at $Q$ depend on $b_Q,d_Q$, whereas
$\Delta_3$ also involves the parameters of other current parents.
All Fourier projections and sup norms on $F_Q$ are taken over
$j=0,\ldots,N-1$.

We assume inductively that
\begin{equation}\label{eq:induction-hypothesis}
 |g_{n-1}(Q)-c_{n-1}|
 \le A N^{-(n-1)},\qquad Q\in K_{n-1},
\end{equation}
where
\[
 A=KN\tau^2
\]
and $K\ge1$ is a fixed constant to be chosen at the end.
We require $N\tau^2$ to be sufficiently small. Since
$\tau=N\lambda$ and $\lambda=N^{-1/s}$, this is possible
for $s<2/3$ by taking $N$ sufficiently large.

Our goal is to find a number $c_n$ such that
\[
 |g_n(x)-c_n|\le AN^{-n},\qquad x\in K_n.
\]
The key observation is
\[
\begin{aligned}
 g_n(x_{Qj})-c_n
 ={}&
 \bigl[g_{n-1}(Q)+N^{-n}\overline F_Q-c_n\bigr]+N^{-n}\bigl[F_Q(j)-\overline F_Q\bigr].
\end{aligned}
\]
Thus we need to control two quantities: the mean error and
the nonconstant oscillation.

The three terms in $F_Q$ play different roles. The local
term $\Delta_1$ allows us to correct both the mean and the first
Fourier mode. The exterior term $\Delta_2$ provides the leading
first-mode error, of size $O(N\tau)$, which must be canceled.
The parent-replacement term $\Delta_3$ will be exponentially small.

We first establish the estimates for these three terms.
We then choose $c_n$, next the scalar parameters $b_Q$, and
finally the shape parameters $d_Q$, checking that the last
correction does not undo the preceding one.

Throughout this section the deformation is quadratic:
\[
 V_Q(\zeta)=\zeta+d_Q\zeta^2.
\]
Higher-order polynomial deformations will be treated separately.

\subsection{The local interaction and the response to $b$ and $d$}

We begin with the local interaction $\Delta_1$. Our aim is to
obtain two different responses: multiplying the sibling
configuration by $b$ changes its mean exactly, whereas the
quadratic deformation produces a first Fourier mode whose
leading term is $(N-2)\Re(d\zeta_j)$.

Recall that
\[
 \zeta_j=e^{2\pi i j/N},\qquad j=0,\ldots,N-1,
\]
and put
\[
 V_d(\zeta)=\zeta+d\zeta^2.
\]
For $b>0$, define the model local interaction
\[
 \Delta_1(j;d,b)
 :=
 \sum_{\substack{0\le k<N\\k\ne j}}
 \log\Bigl(b\,|V_d(\zeta_j)-V_d(\zeta_k)|\Bigr).
\]
Thus $j\mapsto\Delta_1(j;d,b)$ is a real-valued function
on the $N$ roots of unity.

In the actual construction,
\[
 \Delta_1(x_{Qj})
 =(N-1)\log(\rho r_Q)+\Delta_1(j;d_Q,b_Q).
\]
The first term is independent of $j$ and contributes only to
the constant Fourier mode. We use the abbreviations
\[
 \Delta_1(d,b):=\Delta_1(\cdot;d,b),
 \qquad
 \Delta_1(d):=\Delta_1(d,1),
\]
and
\[
 \overline{\Delta_1(d,b)}
 :=\frac1N\sum_{j=0}^{N-1}\Delta_1(j;d,b).
\]
All $\ell^\infty$ norms in the following lemma are taken over
the child index $j$.

\begin{lemma}\label{lem:D1-response}
Let $|d|\le d_0$, where $d_0>0$ is a sufficiently small
absolute constant. Then
\begin{equation}\label{eq:D1-first-mode}
 \left\|
 P_1\Delta_1(d)-(N-2)\Re(d\zeta_j)
 \right\|_\infty
 \le CN|d|^2.
\end{equation}
The remaining nonconstant Fourier modes satisfy
\begin{equation}\label{eq:D1-higher-modes}
 \|P_{\ge2}\Delta_1(d)\|_\infty
 \le CN|d|^2.
\end{equation}
Moreover,
\begin{equation}\label{eq:D1-mean-d}
 \left|
 \overline{\Delta_1(d)}-\overline{\Delta_1(0)}
 \right|
 \le CN(C|d|)^N.
\end{equation}
Finally, for every $b>0$,
\begin{equation}\label{eq:D1-b-scaling}
 \Delta_1(j;d,b)
 =\Delta_1(j;d,1)+(N-1)\log b.
\end{equation}
Consequently, changing $b$ affects only the constant Fourier mode.
\end{lemma}

\begin{proof}
We first fix $b=1$   compare the deformed configuration
with the regular $N$-gon. For $k\ne j$,
\[
\begin{aligned}
 V_d(\zeta_j)-V_d(\zeta_k)
 &=(\zeta_j-\zeta_k)+d(\zeta_j^2-\zeta_k^2)=(\zeta_j-\zeta_k)
   \bigl(1+d(\zeta_j+\zeta_k)\bigr).
\end{aligned}
\]
Hence
\begin{equation}\label{eq:D1-response-log}
\begin{aligned}
 &\Delta_1(j;d,1)-\Delta_1(j;0,1)=\Re\sum_{k\ne j}
   \log\bigl(1+d(\zeta_j+\zeta_k)\bigr).
\end{aligned}
\end{equation}

Since $|\zeta_j+\zeta_k|\le2$ and $|d|\le d_0$,
we may assume
\[
 |d(\zeta_j+\zeta_k)|\le\frac12.
\]
Using
\[
 \log(1+w)=w+O(|w|^2),
\]
we obtain
\begin{equation}\label{eq:D1-response-first-expansion}
 \Delta_1(j;d,1)-\Delta_1(j;0,1)
 =
 \Re\left[d\sum_{k\ne j}(\zeta_j+\zeta_k)\right]+R_j,
\end{equation}
where
\[
 \|R\|_\infty\le CN|d|^2.
\]

By the roots-of-unity identity,
\[
 \sum_{k\ne j}\zeta_k=-\zeta_j,
\]
and therefore
\[
 \sum_{k\ne j}(\zeta_j+\zeta_k)
 =(N-2)\zeta_j.
\]
Thus
\begin{equation}\label{eq:D1-response-main}
 \Delta_1(j;d,1)-\Delta_1(j;0,1)
 =(N-2)\Re(d\zeta_j)+R_j,
 \qquad
 \|R\|_\infty\le CN|d|^2.
\end{equation}

The function $j\mapsto\Re(d\zeta_j)$ belongs exactly to
the first real Fourier mode. On the other hand,
$\Delta_1(j;0,1)$ is independent of $j$. Applying $P_1$ and $P_{\ge2}$ to
\eqref{eq:D1-response-main}, and using the uniform
$\ell^\infty$ bounds for these projections, proves
\eqref{eq:D1-first-mode} and \eqref{eq:D1-higher-modes}.

We next estimate the variation of the mean.
Expanding the logarithm in \eqref{eq:D1-response-log},
\begin{equation}\label{eq:D1-full-log-expansion}
 \Delta_1(j;d,1)-\Delta_1(j;0,1)
 =
 \Re\sum_{m=1}^{\infty}
 \frac{(-1)^{m+1}}m d^m
 \sum_{k\ne j}(\zeta_j+\zeta_k)^m.
\end{equation}
We claim that, after averaging over $j$, every term
with $1\le m<N$ vanishes. Indeed,
\[
\begin{aligned}
 &\sum_j\sum_{k\ne j}(\zeta_j+\zeta_k)^m=
 \sum_{r=0}^m\binom mr
 \sum_j\sum_{k\ne j}\zeta_j^r\zeta_k^{m-r}.
\end{aligned}
\]
If $0<r<m$, then
\[
\begin{aligned}
 \sum_j\sum_{k\ne j}\zeta_j^r\zeta_k^{m-r}
 &=
 \left(\sum_j\zeta_j^r\right)
 \left(\sum_k\zeta_k^{m-r}\right)
 -\sum_j\zeta_j^m=0.
\end{aligned}
\]
The cases $r=0$ and $r=m$ also vanish since
$\sum_j\zeta_j^m=0$ for $1\le m<N$. Consequently,
\begin{equation}\label{eq:D1-mean-cancellation}
 \frac1N\sum_j\sum_{k\ne j}
 (\zeta_j+\zeta_k)^m=0,
 \qquad 1\le m<N.
\end{equation}
Only terms with $m\ge N$ contribute to the mean.
Since $|\zeta_j+\zeta_k|\le2$,
\[
\begin{aligned}
 \left|
 \overline{\Delta_1(d)}-\overline{\Delta_1(0)}
 \right|
 &\le
 \frac1N\sum_j\sum_{k\ne j}
 \sum_{m=N}^{\infty}\frac{(2|d|)^m}{m}\le CN\sum_{m=N}^{\infty}(2|d|)^m\le CN(C|d|)^N.
\end{aligned}
\]
This proves \eqref{eq:D1-mean-d}.

Finally, multiplying all sibling distances by $b$
adds $\log b$ to each of the $N-1$ logarithmic terms.
Hence
\[
 \Delta_1(j;d,b)
 =\Delta_1(j;d,1)+(N-1)\log b,
\]
which proves \eqref{eq:D1-b-scaling}.
\end{proof}

\subsection{The exterior field and the coefficients $A_p(Q)$}

The exterior term $\Delta_2$ measures how the potential created
by the other parents changes when the observation point moves
from $Q$ to one of its children. We first express this change in powers of the child displacement.
The main task is then to estimate the coefficients $A_p(Q)$,
particularly $A_1(Q)$.

Fix $Q\in K_{n-1}$. The geometry of $K_{n-1}$ is determined
by the parameters attached to earlier generations:
\[
 R\in K_0\cup\cdots\cup K_{n-2},
 \qquad
 |\log b_R|\le\varepsilon_0,\quad |d_R|\le C_0\tau.
\]
These previous parameters are already fixed. The new parameters $b_Q,d_Q$ do not enter $A_p(Q)$,
since the coefficients depend only on $K_{n-1}$. For trial parameters at $Q$, write
\[
 x_{Qj}=Q+\rho b_Qr_QV_{d_Q}(\zeta_j),
 \qquad
 h_j=x_{Qj}-Q.
\]
By the separation estimate,
\[
 |y-Q|\gtrsim\frac{r_Q}{\tau},
 \qquad y\in K_{n-1}\setminus\{Q\},
\]
while $|h_j|\lesssim r_Q$. Consequently,
\[
 \left|\frac{h_j}{y-Q}\right|\le C\tau.
\]

For $\tau$ sufficiently small, we may expand
\[
 y-x_{Qj}
 =(y-Q)\left(1-\frac{h_j}{y-Q}\right).
\]
Thus
\[
\begin{aligned}
 \Delta_2(x_{Qj})
 &=N\Re\sum_{y\ne Q}
 \log\left(1-\frac{h_j}{y-Q}\right)=-N\Re\sum_{y\ne Q}\sum_{p\ge1}
 \frac1p\left(\frac{h_j}{y-Q}\right)^p.
\end{aligned}
\]
The series converges absolutely. We obtain
\begin{equation}\label{eq:D2-expansion}
 \Delta_2(x_{Qj})
 =-\Re\sum_{p\ge1}\frac{A_p(Q)}p h_j^p,
 \qquad
 A_p(Q)=N\sum_{\substack{y\in K_{n-1}\\y\ne Q}}
 \frac1{(y-Q)^p}.
\end{equation}

The next lemma gives the estimates needed later.
For $p\ge2$, absolute summation is sufficient.
For $p=1$, we shall use the cancellation of the circular
configuration to avoid a logarithmic loss.

\begin{lemma}[Exterior-field coefficients for polynomial configurations]
\label{lem:Ap}
Fix an integer $M\ge1$. Suppose that, up to generation $n-1$, all
children configurations have the form
\[
 x_{Pj}=x_P+\rho b_P r_P V_P(\zeta_j),\qquad
 V_P(z)=z+\sum_{m=1}^{M}d_{P,m}z^{m+1},
 \qquad \zeta_j=e^{2\pi i j/N},
\]
where $\rho=1/10$, $|\log b_P|\le\varepsilon_0$, and
\[
 |d_{P,m}|\le B_M\tau^m,\qquad 1\le m\le M,
 \qquad \tau=N\lambda.
\]
Here $B_M$ is fixed and the smallness required of $\tau$ may depend on
$M,B_M,\varepsilon_0$. For $Q\in K_{n-1}$ set
\[
 A_p(Q)=N\sum_{\substack{y\in K_{n-1}\\y\ne Q}}
 \frac1{(y-Q)^p},\qquad p\ge1.
\]
There is a constant $C_M\ge1$, independent of $N,n,Q$ (and $p$),
such that
\begin{equation}\label{eq:A1}
 r_Q|A_1(Q)|\le C_M N\tau,
\end{equation}
and, for every integer $p\ge2$,
\begin{equation}\label{eq:Ap}
 r_Q^p|A_p(Q)|\le C_M^p N\tau^p.
\end{equation}
\end{lemma}

\begin{proof}
Throughout the proof $C_M$ is allowed to change from line to line;
its dependence on the fixed coefficient bound $B_M$ is understood.
We assume that $\tau$ is sufficiently small, and also that $N$ is
sufficiently large, depending only on the fixed parameters.

\smallskip
\noindent\emph{Step 1: genealogical coronas and their geometry.}
Fix $Q\in K_{n-1}$. For $1\le\ell\le n-1$, denote by $P_\ell$
the ancestor of $Q$ lying $\ell$ generations above $Q$, and put
\[
 R_\ell=r_{P_\ell}=\lambda^{-\ell}r_Q.
\]
Let $k_0=k_0(\ell,Q)$ be the unique index for which $Q$ descends from
the child $P_\ell k_0$. The centers of the children of $P_\ell$ are
\[
 c_{\ell,k}=x_{P_\ell k}
 =x_{P_\ell}+\rho b_{P_\ell}R_\ell
   V_{P_\ell}(\zeta_k),\qquad 0\le k<N.
\]
For $k\ne k_0$, let $\mathcal B_{\ell,k}$ denote the points of
$K_{n-1}$ which descend from $c_{\ell,k}$. The $\ell$-th corona
is the disjoint union of these $N-1$ branches, each containing
$N^{\ell-1}$ points. In particular,
\[
 \#\mathcal B_{\ell,k}=N^{\ell-1},\qquad
 \#\mathcal Y_\ell(Q)=(N-1)N^{\ell-1}.
\]
By the barycenter identity for each family of descendants,
\begin{equation}\label{eq:Ap-barycenter}
 \sum_{y\in\mathcal B_{\ell,k}}(y-c_{\ell,k})=0.
\end{equation}
Moreover, the cylinder-diameter estimate gives
\begin{equation}\label{eq:Ap-branch-radius}
 |Q-c_{\ell,k_0}|\le C_M\lambda R_\ell,
 \qquad
 \sup_{y\in\mathcal B_{\ell,k}}|y-c_{\ell,k}|
 \le C_M\lambda R_\ell.
\end{equation}

We next justify separation uniformly in $N$. Since
\[
 V_P(z)-z=\sum_{m=1}^M d_{P,m}z^{m+1},\qquad
 \sup_{|z|\le1}|V_P'(z)-1|
 \le\sum_{m=1}^M(m+1)|d_{P,m}|\le C_M\tau,
\]
for any $z,w$ in the closed unit disk the identity
\[
 V_P(z)-V_P(w)
 =(z-w)\int_0^1 V_P'\bigl(w+t(z-w)\bigr)\,dt
\]
shows that
\[
 (1-C_M\tau)|z-w|\le |V_P(z)-V_P(w)|
 \le(1+C_M\tau)|z-w|.
\]
As $b_P$ is bounded above and below by fixed positive constants,
\begin{equation}\label{eq:Ap-sibling-distance}
 |c_{\ell,k}-c_{\ell,k_0}|
 \ge c_M R_\ell|\zeta_k-\zeta_{k_0}|.
\end{equation}
For distinct roots of unity,
\[
 |\zeta_k-\zeta_{k_0}|
 =2\left|\sin\frac{\pi(k-k_0)}N\right|
 \ge \frac{c}{N}.
\]
Combining this inequality with \eqref{eq:Ap-branch-radius},
\eqref{eq:Ap-sibling-distance}, and $N\lambda=\tau$, we get,
for $\tau$ sufficiently small,
\begin{equation}\label{eq:Ap-center-separation}
 |c_{\ell,k}-Q|\ge c_M R_\ell
 |\zeta_k-\zeta_{k_0}|,\qquad k\ne k_0.
\end{equation}
In particular, by \eqref{eq:Ap-branch-radius},
\[
 \frac{|y-c_{\ell,k}|}{|c_{\ell,k}-Q|}
 \le C_M\frac{\lambda}{|\zeta_k-\zeta_{k_0}|}
 \le C_M N\lambda=C_M\tau<\tfrac12
 \quad (y\in\mathcal B_{\ell,k}).
\]

\smallskip
\noindent\emph{Step 2: replacing every branch by its barycenter.}
Let $f_p(z)=(z-Q)^{-p}$, $p\ge1$. For
$y\in\mathcal B_{\ell,k}$ write
\[
 y-Q=(c_{\ell,k}-Q)(1+u_y),\qquad
 u_y=\frac{y-c_{\ell,k}}{c_{\ell,k}-Q},\qquad |u_y|<\tfrac12.
\]
The Taylor estimate
\[
 (1+u)^{-p}=1-pu+O(C^p|u|^2),\qquad |u|\le\tfrac12,
\]
valid uniformly for all integers $p\ge1$ (after enlarging the
absolute constant $C$), gives
\[
 \frac1{(y-Q)^p}
 =\frac1{(c_{\ell,k}-Q)^p}
 -\frac{p(y-c_{\ell,k})}{(c_{\ell,k}-Q)^{p+1}}
 +O\left(\frac{C^p|y-c_{\ell,k}|^2}
 {|c_{\ell,k}-Q|^{p+2}}\right).
\]
Summing over $y\in\mathcal B_{\ell,k}$, the linear term vanishes
by \eqref{eq:Ap-barycenter}. Hence
\begin{equation}\label{eq:Ap-branch-replacement}
 \begin{split}
 \sum_{y\in\mathcal B_{\ell,k}}\frac1{(y-Q)^p}
 &=\frac{N^{\ell-1}}{(c_{\ell,k}-Q)^p}+E_{\ell,k,p},\\
 |E_{\ell,k,p}|
 &\le C_M^p N^{\ell-1}
 \frac{(\lambda R_\ell)^2}{|c_{\ell,k}-Q|^{p+2}}.
 \end{split}
\end{equation}
It follows that $A_p(Q)$ is the sum, over the finitely many
coronas, of the main terms
\[
 A_{p,\ell}^{\mathrm{main}}
 :=N N^{\ell-1}\sum_{k\ne k_0}
 \frac1{(c_{\ell,k}-Q)^p}
\]
and the replacement errors $N\sum_{k\ne k_0}E_{\ell,k,p}$.

\smallskip
\noindent\emph{Step 3: estimating the main terms when $p\ge2$.}
The separation of the roots of unity implies
\begin{equation}\label{eq:Ap-root-sums}
 \sum_{k\ne k_0}|\zeta_k-\zeta_{k_0}|^{-p}
 \le C^pN^p\sum_{j=1}^{\infty}j^{-p}
 \le C^pN^p,\qquad p\ge2.
\end{equation}
Indeed, there are at most two roots at circular index distance $j$,
and $|\zeta_k-\zeta_{k_0}|\ge c\,j/N$ for
$1\le j\le\lfloor N/2\rfloor$.
By \eqref{eq:Ap-center-separation} and \eqref{eq:Ap-root-sums},
\[
 \sum_{k\ne k_0}|c_{\ell,k}-Q|^{-p}
 \le C_M^p N^p R_\ell^{-p}.
\]
Consequently,
\begin{align*}
 r_Q^p|A_{p,\ell}^{\mathrm{main}}|
 &\le C_M^p r_Q^p N N^{\ell-1}N^pR_\ell^{-p}=C_M^p N^{p+\ell}\lambda^{p\ell}=C_M^pN\tau^p(N\lambda^p)^{\ell-1}.
\end{align*}
For $p\ge2$,
\[
 N\lambda^p\le N\lambda^2=\frac{\tau^2}{N}<\frac12.
\]
Thus the sum of all main contributions is bounded by
\begin{equation}\label{eq:Ap-main-p}
 r_Q^p\sum_{\ell\ge1}|A_{p,\ell}^{\mathrm{main}}|
 \le C_M^pN\tau^p.
\end{equation}

\smallskip
\noindent\emph{Step 4: replacement errors when $p\ge2$.}
Using \eqref{eq:Ap-branch-replacement},
\eqref{eq:Ap-center-separation}, and the root sum
\eqref{eq:Ap-root-sums} with exponent $p+2$, we find
\begin{align*}
 r_Q^p N\sum_{k\ne k_0}|E_{\ell,k,p}|
 &\le C_M^p r_Q^p N N^{\ell-1}(\lambda R_\ell)^2
 \sum_{k\ne k_0}|c_{\ell,k}-Q|^{-p-2}\\
 &\le C_M^p r_Q^p N^\ell(\lambda R_\ell)^2
 N^{p+2}R_\ell^{-p-2}= C_M^p N^{p+\ell+2}\lambda^{p\ell+2}= C_M^p\tau^2 N\tau^p(N\lambda^p)^{\ell-1}.
\end{align*}
The extra factor $\tau^2$ makes these errors smaller than the
corresponding main bounds. Summing in $\ell$ and combining with
\eqref{eq:Ap-main-p} proves \eqref{eq:Ap}.

\smallskip
\noindent\emph{Step 5: the circular cancellation for $p=1$.}
Here the absolute-value bound
\[
 \sum_{k\ne k_0}|\zeta_k-\zeta_{k_0}|^{-1}
 \le CN\log N
\]
would produce a loss. We therefore exploit the exact identity
\begin{equation}\label{eq:A1-root-identity}
 \sum_{k\ne k_0}\frac1{\zeta_k-\zeta_{k_0}}
 =-\frac{N-1}{2\zeta_{k_0}}.
\end{equation}
For completeness, let $P(z)=z^N-1$. At its simple zero
$a=\zeta_{k_0}$, the logarithmic derivative of
$P(z)/(z-a)$ gives
\[
 \sum_{k\ne k_0}\frac1{a-\zeta_k}
 =\frac{P''(a)}{2P'(a)}=\frac{N-1}{2a},
\]
which is equivalent to \eqref{eq:A1-root-identity}.

We show that this cancellation persists for the polynomially
perturbed configuration. Abbreviate $V=V_{P_\ell}$ and
$a=\zeta_{k_0}$, and define the divided difference
\[
 D_a(z)=\frac{V(z)-V(a)}{z-a}
 =1+\sum_{m=1}^M d_{P_\ell,m}
     \sum_{\nu=0}^m z^{m-\nu}a^\nu.
\]
It extends continuously to $z=a$, where $D_a(a)=V'(a)$.
Since $M$ is fixed and $|d_{P_\ell,m}|\le B_M\tau^m$,
\begin{equation}\label{eq:A1-divided-bounds}
 \sup_{|z|,|a|\le1}|D_a(z)-1|\le C_M\tau,
 \qquad
 \sup_{|z|,|a|\le1}|\partial_zD_a(z)|\le C_M\tau.
\end{equation}
In particular, $|D_a(z)|\ge1/2$ on the closed disk, and
$|D_a(a)|=|V'(a)|\ge1/2$.
Integrating the derivative of $D_a$ along the segment from $a$
to $z$ gives
\[
 |D_a(z)-D_a(a)|\le C_M\tau|z-a|.
\]
Therefore
\[
 \left|\frac1{D_a(z)}-\frac1{D_a(a)}\right|
 =\frac{|D_a(a)-D_a(z)|}{|D_a(z)D_a(a)|}
 \le C_M\tau|z-a|.
\]
Dividing by $z-a$ yields the useful pointwise identity
\begin{equation}\label{eq:A1-polynomial-difference-quotient}
 \frac1{V(z)-V(a)}
 =\frac1{V'(a)(z-a)}+O_M(\tau),
 \qquad z\ne a,\quad |z|=|a|=1.
\end{equation}
The constant in $O_M(\tau)$ is independent of $z,a,N$;
importantly, there is no factor $|z-a|^{-1}$ in the error.
Summing over $z=\zeta_k$, $k\ne k_0$, using
\eqref{eq:A1-root-identity} and $|V'(a)|\ge1/2$, gives
\begin{align*}
 \left|\sum_{k\ne k_0}
 \frac1{V(\zeta_k)-V(a)}\right|
 &\le\frac1{|V'(a)|}
 \left|\sum_{k\ne k_0}\frac1{\zeta_k-a}\right|
 +C_M(N-1)\tau\le C_MN.
\end{align*}
Because
$c_{\ell,k}-c_{\ell,k_0}
 =\rho b_{P_\ell}R_\ell
 (V(\zeta_k)-V(a))$, we conclude that
\begin{equation}\label{eq:A1-center-circular}
 \left|\sum_{k\ne k_0}
 \frac1{c_{\ell,k}-c_{\ell,k_0}}\right|
 \le\frac{C_MN}{R_\ell}.
\end{equation}

\smallskip
\noindent\emph{Step 6: replacing the distinguished center by $Q$.}
For $k\ne k_0$,
\[
 \frac1{c_{\ell,k}-Q}-
 \frac1{c_{\ell,k}-c_{\ell,k_0}}
 =\frac{Q-c_{\ell,k_0}}
 {(c_{\ell,k}-Q)(c_{\ell,k}-c_{\ell,k_0})}.
\]
By \eqref{eq:Ap-branch-radius},
\eqref{eq:Ap-center-separation}, and
\eqref{eq:Ap-sibling-distance}, its modulus is bounded by
\[
 \frac{C_M\lambda R_\ell}
 {R_\ell^2|\zeta_k-a|^2}.
\]
Using \eqref{eq:Ap-root-sums} with $p=2$,
\begin{align*}
 \sum_{k\ne k_0}\left|
 \frac1{c_{\ell,k}-Q}-
 \frac1{c_{\ell,k}-c_{\ell,k_0}}\right|
 &\le\frac{C_M\lambda}{R_\ell}
 \sum_{k\ne k_0}|\zeta_k-a|^{-2}\le\frac{C_M\lambda N^2}{R_\ell}
 =\frac{C_MN\tau}{R_\ell}
 \le\frac{C_MN}{R_\ell}.
\end{align*}
Combining this with \eqref{eq:A1-center-circular} gives
\begin{equation}\label{eq:A1-perturbed-circular}
 \left|\sum_{k\ne k_0}\frac1{c_{\ell,k}-Q}\right|
 \le\frac{C_MN}{R_\ell}.
\end{equation}
Thus, for the main term of the $\ell$-th corona,
\begin{align*}
 r_Q|A_{1,\ell}^{\mathrm{main}}|
 &\le C_M r_Q N N^{\ell-1}\frac{N}{R_\ell}=C_MN^{\ell+1}\lambda^\ell
 =C_MN\tau^\ell.
\end{align*}
Since $\tau<1/2$, summing over $\ell\ge1$ yields
\begin{equation}\label{eq:A1-main-summed}
 r_Q\sum_{\ell\ge1}|A_{1,\ell}^{\mathrm{main}}|
 \le C_MN\sum_{\ell\ge1}\tau^\ell
 \le C_MN\tau.
\end{equation}

\smallskip
\noindent\emph{Step 7: replacement errors for $p=1$.}
The estimate \eqref{eq:Ap-branch-replacement}, now with $p=1$,
and \eqref{eq:Ap-root-sums} with exponent $3$ yield
\begin{align*}
 r_Q N\sum_{k\ne k_0}|E_{\ell,k,1}|
 &\le C_Mr_Q N N^{\ell-1}(\lambda R_\ell)^2
 \sum_{k\ne k_0}|c_{\ell,k}-Q|^{-3}\\
 &\le C_Mr_Q N^\ell\lambda^2R_\ell^2
 \frac{N^3}{R_\ell^3}=C_MN^{\ell+3}\lambda^{\ell+2}
 =C_MN\tau^\ell\tau^2.
\end{align*}
Summing over the coronas gives
\[
 r_QN\sum_{\ell\ge1}\sum_{k\ne k_0}|E_{\ell,k,1}|
 \le C_MN\tau^2\sum_{\ell\ge1}\tau^\ell
 \le C_MN\tau^3.
\]
Combining this with \eqref{eq:A1-main-summed} proves
\eqref{eq:A1} and completes the proof.
\end{proof}

\begin{corollary}\label{cor:D2}
Let $Q\in K_{n-1}$ and let $b_Q,d_Q$ be admissible.
Write
\[
 \Delta_2(j):=\Delta_2(x_{Qj}),\qquad 0\le j<N.
\]
Then
\begin{equation}\label{eq:D2-mode-bounds}
 \|P_1\Delta_2\|_\infty\le CN\tau,
 \qquad
 \|P_{\ge2}\Delta_2\|_\infty\le CN\tau^2.
\end{equation}
Furthermore,
\begin{equation}\label{eq:D2-mean-bound}
 |\overline{\Delta_2(d_Q,b_Q)}|
 \le CN(C\tau)^{N/2}.
\end{equation}
In the unperturbed case $d_Q=0$,
\begin{equation}\label{eq:D2-mean-unperturbed}
 |\overline{\Delta_2(0,b_Q)}|
 \le CN(C\tau)^N.
\end{equation}
In particular,
\begin{equation}\label{eq:D2-mean-d-stability}
 \left|
 \overline{\Delta_2(d_Q,b_Q)}
 -\overline{\Delta_2(0,b_Q)}
 \right|
 \le CN(C\tau)^{N/2}.
\end{equation}
\end{corollary}

\begin{proof}
Put
\[
 \alpha_Q:=\rho b_Q,\qquad
 B_p:=\frac{A_p(Q)}p(\alpha_Qr_Q)^p.
\]
By Lemma~\ref{lem:Ap},
\begin{equation}\label{eq:D2-Bp-bound}
 |B_p|\le\frac{CN}{p}(C\tau)^p,\qquad p\ge1.
\end{equation}
The expansion \eqref{eq:D2-expansion} becomes
\begin{equation}\label{eq:D2-series-cor}
 \Delta_2(j)
 =-\Re\sum_{p\ge1}B_p(\zeta_j+d_Q\zeta_j^2)^p.
\end{equation}

We first estimate the nonconstant Fourier modes.
Write
\[
 \Delta_2(j)=-\Re(B_1\zeta_j)+\mathcal R(j),
\]
where
\[
 \mathcal R(j)
 =-\Re\left[
 B_1d_Q\zeta_j^2
 +\sum_{p\ge2}B_p(\zeta_j+d_Q\zeta_j^2)^p
 \right].
\]
Since $|d_Q|\le C_0\tau$,
\[
\begin{aligned}
 \|\mathcal R\|_\infty
 &\le |B_1||d_Q|
 +\sum_{p\ge2}|B_p|(1+|d_Q|)^p\le CN\tau^2+CN\sum_{p\ge2}(C\tau)^p\le CN\tau^2.
\end{aligned}
\]
The term $\Re(B_1\zeta_j)$ belongs to the first Fourier mode.
The uniform bounds for the projections $P_1$ and $P_{\ge2}$
now give \eqref{eq:D2-mode-bounds}.

For the mean, expand
\[
 (\zeta_j+d_Q\zeta_j^2)^p
 =\sum_{a=0}^p\binom pa d_Q^a\zeta_j^{p+a}.
\]
Since
\[
 \frac1N\sum_{j=0}^{N-1}\zeta_j^m
 =
 \begin{cases}
 1,&N\mid m,\\
 0,&N\nmid m,
 \end{cases}
\]
we obtain
\begin{equation}\label{eq:D2-mean-filter}
 \overline{\Delta_2(d_Q,b_Q)}
 =-\Re\sum_{p\ge1}B_p
 \sum_{\substack{0\le a\le p\\N\mid(p+a)}}
 \binom pa d_Q^a.
\end{equation}
Because $p\le p+a\le2p$, no term with $p<N/2$
can survive the averaging.

Consequently,
\[
\begin{aligned}
 |\overline{\Delta_2(d_Q,b_Q)}|
 &\le CN\sum_{p\ge\lceil N/2\rceil}
 (C\tau)^p(1+|d_Q|)^p\le CN(C\tau)^{N/2}.
\end{aligned}
\]
This gives \eqref{eq:D2-mean-bound}.

If $d_Q=0$, only $a=0$ remains and $N\mid p$
is necessary. Thus
\[
 |\overline{\Delta_2(0,b_Q)}|
 \le CN\sum_{q\ge1}(C\tau)^{qN}
 \le CN(C\tau)^N.
\]
Finally, \eqref{eq:D2-mean-d-stability} follows from
the triangle inequality.
\end{proof}

\subsection{The parent-replacement term $\Delta_3$}

The last term $\Delta_3$ records the error from replacing
each exterior parent $y$ by its $N$ children.

Unlike $\Delta_2$, this term depends on the new parameters
at all parents $y\in K_{n-1}$. We prove that its entire
supremum norm is exponentially small. Thus no separate
Fourier analysis of $\Delta_3$ will be needed.

Recall the parameter vectors
\[
 \mathbf d=(d_y)_{y\in K_{n-1}},
 \qquad
 \mathbf t=(t_y)_{y\in K_{n-1}}.
\]

\begin{lemma}\label{lem:D3}
Assume
\[
 \|\mathbf d\|_{\ell^\infty(K_{n-1})}\le C_0\tau,
 \qquad
 \|\mathbf t\|_{\ell^\infty(K_{n-1})}\le\varepsilon_0.
\]
Then
\[
 \|\Delta_3\|_\infty\le CN(C\tau)^{N/2},
\]
where
\[
 \|\Delta_3\|_\infty
 :=
 \max_{Q\in K_{n-1}}\max_{0\le j<N}
 |\Delta_3(x_{Qj})|.
\]
The constant is independent of the generation and of the
admissible parameters.
\end{lemma}

\begin{proof}
Fix $Q\in K_{n-1}$ and one of its children
$x=x_{Qj}\in K_n$.

By definition,
\[
 \Delta_3(x)
 =\sum_{\substack{y\in K_{n-1}\\y\ne Q}}D_y(x),
\]
where
\[
 D_y(x)
 =\sum_{k=0}^{N-1}
 \bigl(\log|y+h_{y,k}-x|-\log|y-x|\bigr)
\]
and
\[
 h_{y,k}
 =\rho b_yr_y(\zeta_k+d_y\zeta_k^2)
 =\rho b_yr_y\zeta_k(1+d_y\zeta_k).
\]
Since $b_y\simeq1$ and $|d_y|\le C_0\tau$,
\begin{equation}\label{eq:D3-h-size}
 \max_k|h_{y,k}|\le Cr_y.
\end{equation}

We first record the cancellation within one packet of children.
For every integer $q\ge1$,
\[
\begin{aligned}
 \sum_{k=0}^{N-1}h_{y,k}^q
 &=(\rho b_yr_y)^q
 \sum_{k=0}^{N-1}\zeta_k^q(1+d_y\zeta_k)^q=(\rho b_yr_y)^q
 \sum_{a=0}^q\binom qa d_y^a
 \sum_{k=0}^{N-1}\zeta_k^{q+a}.
\end{aligned}
\]
If $1\le q<N/2$, then
\[
 1\le q+a\le2q<N.
\]
Therefore every inner roots-of-unity sum vanishes, and
\begin{equation}\label{eq:D3-moment-cancel}
 \sum_{k=0}^{N-1}h_{y,k}^q=0,
 \qquad 1\le q<\frac N2.
\end{equation}

We next show that $h_{y,k}$ is small compared with $|x-y|$.

Let $\mathcal Y_\ell(Q)$ consist of the vertices
$y\in K_{n-1}$ whose branches first separate from that
of $Q$ at its $\ell$-th ancestor. Put
\[
 R_\ell=\lambda^{-\ell}r_Q.
\]
For $y\in\mathcal Y_\ell(Q)$, the separation estimate gives
\[
 |Q-y|\ge c\frac{R_\ell}{N}.
\]
On the other hand,
\[
 |x-Q|\le Cr_Q=C\lambda^\ell R_\ell.
\]
Since $N\lambda^\ell\le\tau\ll1$,
\[
 |x-y|\ge c\frac{R_\ell}{N}.
\]
Also $r_y=r_Q=\lambda^\ell R_\ell$.

Define
\[
 \eta_y:=\frac{\max_k|h_{y,k}|}{|x-y|}.
\]
By \eqref{eq:D3-h-size},
\begin{equation}\label{eq:D3-eta-ring}
 \eta_y\le CN\lambda^\ell
 =C\tau\lambda^{\ell-1},
 \qquad y\in\mathcal Y_\ell(Q).
\end{equation}
In particular, $\eta_y\le1/2$ when $\tau$ is sufficiently small.

Now put $w=x-y$. Since
\[
 y+h_{y,k}-x
 =-w\left(1-\frac{h_{y,k}}w\right),
\]
we may expand
\[
\begin{aligned}
 D_y(x)
 &=\Re\sum_{k=0}^{N-1}
 \log\left(1-\frac{h_{y,k}}w\right)=-\Re\sum_{q\ge1}\frac1{qw^q}
 \sum_{k=0}^{N-1}h_{y,k}^q.
\end{aligned}
\]

Let
\[
 L_N:=\lceil N/2\rceil.
\]
By \eqref{eq:D3-moment-cancel}, all the terms
with $q<L_N$ vanish exactly. Thus
\[
 D_y(x)
 =-\Re\sum_{q\ge L_N}\frac1{qw^q}
 \sum_{k=0}^{N-1}h_{y,k}^q.
\]
Consequently,
\[
\begin{aligned}
 |D_y(x)|
 &\le N\sum_{q\ge L_N}\frac{\eta_y^q}{q}\le\frac{N}{L_N}
 \frac{\eta_y^{L_N}}{1-\eta_y}\le C\eta_y^{L_N}.
\end{aligned}
\]

Finally, each genealogical corona contains at most $N^\ell$ vertices:
\[
 \#\mathcal Y_\ell(Q)
 =(N-1)N^{\ell-1}\le N^\ell.
\]
Combining this with \eqref{eq:D3-eta-ring},
\[
\begin{aligned}
 |\Delta_3(x)|
 &\le C\sum_{\ell\ge1}
 N^\ell(C\tau\lambda^{\ell-1})^{L_N}\\
 &=CN(C\tau)^{L_N}
 \sum_{\ell\ge1}(N\lambda^{L_N})^{\ell-1}.
\end{aligned}
\]
For $N$ sufficiently large,
\[
 N\lambda^{L_N}\le\frac12.
\]
Therefore
\[
 |\Delta_3(x)|
 \le CN(C\tau)^{L_N}
 \le CN(C\tau)^{N/2}.
\]
This completes the proof.
\end{proof}

\subsection{Choosing the target level and correcting the mean}

We now use the three estimates to choose the new parameters.
We first construct a common target level $c_n$, then use $b_Q$
to correct the mean. Throughout this step all new shape
parameters are set to zero.

Set $\mathbf d=\mathbf0$ and $\mathbf t=\mathbf0$.
Define
\[
 H_Q:=g_{n-1}(Q)+N^{-n}\overline F_Q(\mathbf0,\mathbf0).
\]
This is the reference mean associated with $Q$ before
any new correction is introduced.

\begin{lemma}\label{lem:mean-values-comparable}
There is a number $\gamma_n$, depending only on $n$, such that
\[
 |H_Q-c_{n-1}-\gamma_n|
 \le CA N^{-(n-1)},
 \qquad Q\in K_{n-1}.
\]
Consequently,
\[
 \operatorname{osc}_{Q\in K_{n-1}}H_Q
 \le CA N^{-(n-1)}.
\]
\end{lemma}

\begin{proof}
At the reference configuration,
\[
 x_{Qj}=Q+\rho r_Q\zeta_j.
\]
Therefore
\[
\begin{aligned}
 \Delta_1(x_{Qj})
 &=\sum_{k\ne j}
 \log|\rho r_Q(\zeta_k-\zeta_j)|\\
 &=(N-1)\log(\rho r_Q)
 +\sum_{k\ne j}\log|\zeta_k-\zeta_j|.
\end{aligned}
\]
The roots-of-unity identity
\[
 \prod_{k\ne j}|\zeta_k-\zeta_j|=N
\]
gives
\[
 \sum_{k\ne j}\log|\zeta_k-\zeta_j|=\log N.
\]
Also $r_Q=\lambda^{n-1}$ for all $Q\in K_{n-1}$.
Hence the entire local term is independent of $Q$ and $j$.

Set
\begin{equation}\label{eq:Gamma-n}
 \Gamma_n
 :=(N-1)\log(\rho\lambda^{n-1})+\log N
\end{equation}
and
\begin{equation}\label{eq:gamma-n}
 \gamma_n:=N^{-n}\Gamma_n.
\end{equation}
By Corollary~\ref{cor:D2} and Lemma~\ref{lem:D3},
\[
 \left|
 \overline{\Delta_{2,Q}}(\mathbf0,\mathbf0)
 +\overline{\Delta_{3,Q}}(\mathbf0,\mathbf0)
 \right|
 \le CN(C\tau)^{N/2}.
\]
Consequently,
\[
 H_Q
 =g_{n-1}(Q)+\gamma_n
 +O\left(N^{-n}N(C\tau)^{N/2}\right).
\]
Using \eqref{eq:induction-hypothesis},
\[
\begin{aligned}
 |H_Q-c_{n-1}-\gamma_n|
 &\le AN^{-(n-1)}
 +CN^{-n}N(C\tau)^{N/2}\le CA N^{-(n-1)}.
\end{aligned}
\]
Taking the difference for two parents gives
the oscillation estimate.
\end{proof}

We now choose the target level for generation $n$ by
\begin{equation}\label{eq:cn-choice}
 c_n:=c_{n-1}+\gamma_n.
\end{equation}
With this choice, the preceding lemma immediately gives
\begin{equation}\label{eq:HQ-cn}
 |H_Q-c_n|\le CA N^{-(n-1)},
 \qquad Q\in K_{n-1}.
\end{equation}

\paragraph{Choosing the scalar parameters.}

Keep $\mathbf d=\mathbf0$.
Since a common dilation by $b_Q=e^{t_Q}$ changes each of the
$N-1$ local logarithmic distances by $t_Q$,
\[
 \overline{\Delta_{1,Q}}(0,t_Q)
 =\overline{\Delta_{1,Q}}(0,0)+(N-1)t_Q.
\]
We therefore define
\begin{equation}\label{eq:tQ-explicit}
 t_Q:=\frac{N^n(c_n-H_Q)}{N-1},
 \qquad b_Q:=e^{t_Q}.
\end{equation}
By \eqref{eq:HQ-cn},
\[
 |t_Q|\le CA.
\]
In particular, if $N\tau^2$ is sufficiently small,
then $|t_Q|\le\varepsilon_0$. Moreover,
\[
 H_Q-c_n+N^{-n}(N-1)t_Q=0.
\]
Thus the local scalar correction removes the reference-mean
error exactly. The only remaining error comes from changes
in $\Delta_2$ and $\Delta_3$.

\begin{lemma}\label{lem:scalar-correction}
For the scalar parameters in \eqref{eq:tQ-explicit}
and $\mathbf d=\mathbf0$,
\[
 \left|
 g_{n-1}(Q)+N^{-n}\overline F_Q(\mathbf0,\mathbf t)-c_n
 \right|
 \le CN^{-n}N(C\tau)^{N/2},
 \qquad Q\in K_{n-1}.
\]
\end{lemma}

\begin{proof}
Subtract the reference mean and use the exact cancellation:
\[
\begin{aligned}
 &g_{n-1}(Q)+N^{-n}\overline F_Q(\mathbf0,\mathbf t)-c_n=N^{-n}\Bigl[
 \overline{\Delta_{2,Q}}(\mathbf0,\mathbf t)
 -\overline{\Delta_{2,Q}}(\mathbf0,\mathbf0)+
 \overline{\Delta_{3,Q}}(\mathbf0,\mathbf t)
 -\overline{\Delta_{3,Q}}(\mathbf0,\mathbf0)
 \Bigr].
\end{aligned}
\]
By Corollary~\ref{cor:D2}, with $d_Q=0$,
\[
 |\overline{\Delta_{2,Q}}(\mathbf0,\mathbf t)|
 +|\overline{\Delta_{2,Q}}(\mathbf0,\mathbf0)|
 \le CN(C\tau)^N.
\]
By Lemma~\ref{lem:D3},
\[
 |\overline{\Delta_{3,Q}}(\mathbf0,\mathbf t)|
 +|\overline{\Delta_{3,Q}}(\mathbf0,\mathbf0)|
 \le CN(C\tau)^{N/2}.
\]
Combining these estimates proves the result.
\end{proof}

\subsection{Correcting the first Fourier mode}

We now keep all scalar parameters $\mathbf t$ fixed and
start from $\mathbf d=\mathbf0$.

Since the unperturbed local interaction is constant in $j$,
the first Fourier mode of $F_Q$ comes only from the exterior
terms. Write
\begin{equation}\label{eq:first-mode-BQ}
 P_1F_Q(\mathbf0,\mathbf t;j)
 =\Re(B_Q\zeta_j)
\end{equation}
for a complex number $B_Q$. By Corollary~\ref{cor:D2} and Lemma~\ref{lem:D3},
\[
 |B_Q|\le CN\tau.
\]
Lemma~\ref{lem:D1-response} shows that introducing $d_Q$
produces the first-mode response
\[
 (N-2)\Re(d_Q\zeta_j)
\]
in $\Delta_1$.

We therefore choose
\begin{equation}\label{eq:dQ-explicit}
 d_Q:=-\frac{B_Q}{N-2},
\end{equation}
so that
\begin{equation}\label{eq:dQ-size-proof}
 |d_Q|\le C\tau.
\end{equation}

\begin{lemma}\label{lem:first-correction}
For the choices \eqref{eq:tQ-explicit}
and \eqref{eq:dQ-explicit},
\begin{equation}\label{eq:first-mode-residual}
 \|P_1F_Q(\mathbf d,\mathbf t)\|_\infty
 \le CN\tau^2,
\end{equation}
and
\begin{equation}\label{eq:oscF}
 \|F_Q(\mathbf d,\mathbf t)
 -\overline F_Q(\mathbf d,\mathbf t)\|_\infty
 \le CN\tau^2.
\end{equation}
Both estimates are uniform in $Q$ and the generation.
\end{lemma}

\begin{proof}
We separate the principal cancellation from the errors
created by the shape correction. By Lemma~\ref{lem:D1-response},
\begin{equation}\label{eq:first-correction-D1}
 P_1\Delta_1(d_Q,b_Q)(j)
 =(N-2)\Re(d_Q\zeta_j)+R_{1,Q}(j),
 \qquad
 \|R_{1,Q}\|_\infty\le CN|d_Q|^2\le CN\tau^2.
\end{equation}
By our choice of $d_Q$,
\[
 (N-2)\Re(d_Q\zeta_j)
 =-\Re(B_Q\zeta_j)
 =-P_1F_Q(\mathbf0,\mathbf t;j).
\]
Thus the leading first-mode response of $\Delta_1$
cancels the entire first Fourier mode of the reference
configuration.

We must check that changing the shape parameters does
not produce another first-mode contribution of order $N\tau$. For $\Delta_2$, the coefficients $A_p(Q)$ are independent
of the new shape parameters. Put $a_Q=\rho b_Q$. Then
\[
\begin{aligned}
 &\Delta_{2,Q}(d_Q,b_Q;j)-\Delta_{2,Q}(0,b_Q;j)=-\Re\sum_{p\ge1}
 \frac{A_p(Q)}p(a_Qr_Q)^p
 \left[
 (\zeta_j+d_Q\zeta_j^2)^p-\zeta_j^p
 \right].
\end{aligned}
\]
We use the elementary estimate
\[
 |(\zeta+d\zeta^2)^p-\zeta^p|
 \le p|d|(1+|d|)^{p-1},
 \qquad |\zeta|=1.
\]
By Lemma~\ref{lem:Ap} and $|d_Q|\le C\tau$,
\[
\begin{aligned}
 &\|\Delta_{2,Q}(d_Q,b_Q)
 -\Delta_{2,Q}(0,b_Q)\|_\infty\le CN|d_Q|
 \sum_{p\ge1}(C\tau)^p \le CN\tau|d_Q|
 \le CN\tau^2.
\end{aligned}
\]
No Fourier expansion of this difference is needed:
an $\ell^\infty$ bound already controls its first mode.

For $\Delta_3$, we use the uniform estimate of
Lemma~\ref{lem:D3} before and after changing all shape
parameters. Since $P_1$ is bounded on discrete $\ell^\infty$,
\begin{equation}\label{eq:first-correction-D3}
 \left\|
 P_1\bigl[
 \Delta_{3,Q}(\mathbf d,\mathbf t)
 -\Delta_{3,Q}(\mathbf0,\mathbf t)
 \bigr]
 \right\|_\infty
 \le CN(C\tau)^{N/2}.
\end{equation}
Combining the preceding estimates with the exact
first-mode cancellation gives
\[
 \|P_1F_Q(\mathbf d,\mathbf t)\|_\infty
 \le CN\tau^2+CN(C\tau)^{N/2}
 \le CN\tau^2.
\]
This proves \eqref{eq:first-mode-residual}.

For the remaining nonconstant modes, the preceding
lemmas give
\[
\begin{aligned}
 \|P_{\ge2}\Delta_1\|_\infty
 &\le CN|d_Q|^2\le CN\tau^2,\\
 \|P_{\ge2}\Delta_2\|_\infty
 &\le CN\tau^2,\\
 \|P_{\ge2}\Delta_3\|_\infty
 &\le CN(C\tau)^{N/2}.
\end{aligned}
\]
Since
\[
 F_Q-\overline F_Q=P_1F_Q+P_{\ge2}F_Q,
\]
we obtain \eqref{eq:oscF}.
\end{proof}

\subsection{Mean stability and closing the induction}

\paragraph{Stability of the mean.}

The scalar correction was chosen while $\mathbf d=\mathbf0$.
We must therefore check that the subsequent choice of
$\mathbf d$ does not undo the correction of the mean.

\begin{lemma}\label{lem:mean-stability-after-d}
For the admissible parameters chosen above,
\[
 \left|
 \overline F_Q(\mathbf d,\mathbf t)
 -\overline F_Q(\mathbf0,\mathbf t)
 \right|
 \le CN(C\tau)^{N/2}.
\]
\end{lemma}

\begin{proof}
We compare the three terms in $F_Q$ separately. Lemma~\ref{lem:D1-response} gives
\[
 \left|
 \overline{\Delta_{1,Q}}(d_Q,b_Q)
 -\overline{\Delta_{1,Q}}(0,b_Q)
 \right|
 \le CN(C\tau)^N.
\]
Corollary~\ref{cor:D2} gives
\[
 \left|
 \overline{\Delta_{2,Q}}(d_Q,b_Q)
 -\overline{\Delta_{2,Q}}(0,b_Q)
 \right|
 \le CN(C\tau)^{N/2}.
\]
Finally, Lemma~\ref{lem:D3} gives
\[
\begin{aligned}
 &\left|
 \overline{\Delta_{3,Q}}(\mathbf d,\mathbf t)
 -\overline{\Delta_{3,Q}}(\mathbf0,\mathbf t)
 \right|\\
 &\quad\le
 \|\Delta_{3,Q}(\mathbf d,\mathbf t)\|_\infty
 +\|\Delta_{3,Q}(\mathbf0,\mathbf t)\|_\infty\\
 &\quad\le CN(C\tau)^{N/2}.
\end{aligned}
\]
Combining these three estimates proves the lemma.
\end{proof}

\paragraph{Closing the induction.}

We have now obtained the two estimates needed at
generation $n$: the mean is close to $c_n$ and the
oscillation around the mean is small.

\begin{proposition}[Inductive step]\label{prop:induction}
There exist absolute constants $K,C>0$ such that,
if
\[
 A=KN\tau^2
\]
and $N\tau^2$ is sufficiently small, then the
inductive hypothesis \eqref{eq:induction-hypothesis}
implies that the parameters chosen in
\eqref{eq:tQ-explicit} and \eqref{eq:dQ-explicit}
satisfy
\[
 |\log b_Q|\le CA,\qquad |d_Q|\le C\tau,
\]
and
\[
 |g_n(x)-c_n|\le AN^{-n},
 \qquad x\in K_n.
\]
\end{proposition}

\begin{proof}
The parameter bounds follow from
\eqref{eq:HQ-cn}, \eqref{eq:tQ-explicit},
and \eqref{eq:dQ-size-proof}. 

Combining Lemma~\ref{lem:scalar-correction}
with Lemma~\ref{lem:mean-stability-after-d}, we obtain
\[
 \left|
 g_{n-1}(Q)+N^{-n}\overline F_Q(\mathbf d,\mathbf t)-c_n
 \right|
 \le CN^{-n}N(C\tau)^{N/2}.
\]

For $x=x_{Qj}$, write
\[
\begin{aligned}
 g_n(x)-c_n
 ={}&
 \bigl[
 g_{n-1}(Q)+N^{-n}\overline F_Q-c_n
 \bigr]\\
 &+N^{-n}(F_Q(j)-\overline F_Q).
\end{aligned}
\]
The first bracket is controlled by the mean estimate,
and the second term by \eqref{eq:oscF}. Consequently,
\[
\begin{aligned}
 |g_n(x)-c_n|
 &\le CN^{-n}N
 \left[(C\tau)^{N/2}+\tau^2\right]\\
 &\le C_1N^{-n}N\tau^2.
\end{aligned}
\]
Choose $K\ge C_1$ in $A=KN\tau^2$, and take
$N\tau^2$ sufficiently small that $CA\le\varepsilon_0$.
This closes the induction.

For the initial step, take $K_0=\{0\}$ and $c_0=0$.
Then $g_0(0)=0$ and there are no exterior families.
Thus the induction holds for every generation.
\end{proof}

From the above, it follows that
\[
 \sup_{x\in K_n}|g_n(x)-c_n|\le AN^{-n}
\]
for every generation.

Finally, we verify that the target levels $c_n$ converge.
By our choice \eqref{eq:cn-choice},
\[
 c_n-c_{n-1}=\gamma_n
\]
exactly. Moreover, \eqref{eq:gamma-n} gives
\[
 |\gamma_n|
 \le C\bigl(1+n|\log\lambda|+\log N\bigr)N^{-(n-1)}.
\]
This bound is summable in $n$, with $N$ and $\lambda$ fixed.
Hence
\[
 \sum_{n\ge1}|c_n-c_{n-1}|<\infty,
\]
and therefore $c_n$ converges to some
$c_\infty\in\mathbb R$.


\section{Polynomial corrections of arbitrary order and the full range}
\label{sec:general-polynomial-correction}

Fix $0<s<1$ and an integer $M\geq1$.
Throughout this section, $s$ and $M$ are fixed, and $N$
will be taken sufficiently large.

We retain the recursive construction and the notation of
Section~\ref{subsec:first-order-parameters}, but replace
the quadratic deformation by a polynomial deformation
of order $M$. At every parent $Q$, we use
\[
 V_Q(\zeta)
 =\zeta+\sum_{r=1}^M d_{Q,r}\zeta^{r+1},
 \qquad |d_{Q,r}|\le C_M\tau^r.
\]
The coefficients $d_{Q,r}$ will be chosen inductively.
We assign weight $\operatorname{wt}(d_{Q,r})=r$
and extend this additively to monomials.

Unlike the first-order construction, we do not assume
$N\tau^2\le\varepsilon_0$.
The geometric separation estimates require only that
$\tau$ and the polynomial deformations be sufficiently small.
Moreover, Lemma~\ref{lem:Ap}
remains valid in the present range $0<s<1$:
its proof only requires $\tau=N^{1-1/s}$ to be
sufficiently small and does not rely on
$N\tau^2\ll1$.
\subsection{The local term}

Put
\begin{equation}\label{eq:Sr-definition}
 S_r(z,w)=\frac{z^{r+1}-w^{r+1}}{z-w}
 =\sum_{\nu=0}^r z^{r-\nu}w^\nu.
\end{equation}
Then
\[
 V_Q(z)-V_Q(w)=(z-w)\left(1+\sum_{r=1}^M d_{Q,r}S_r(z,w)\right).
\]

\begin{lemma}\label{lem:D1-general}
Assume that
\[
    |d_{Q,r}|\leq C_M\tau^r,
    \qquad 1\leq r\leq M.
\]
Then, for every $1\leq k\leq M$,
\[
\begin{split}
    P_k\Delta_1
    &=
    \Re\left[
       \left(
          (N-k-1)d_{Q,k}
          +
          \mathcal P_{N,k}
             (d_{Q,1},\ldots,d_{Q,k-1})
       \right)\zeta_j^k
    \right]
    +
    O_M\left(
       N(C_M\tau)^{N-k}
    \right),
\end{split}
\]
where $\mathcal P_{N,k}$ is a polynomial whose coefficients
depend only on $N$ and $k$, and every monomial of
$\mathcal P_{N,k}$ has weighted degree exactly $k$.
Moreover,
\begin{equation}\label{eq:PNk-size}
    \left|
       \mathcal P_{N,k}
       (d_{Q,1},\ldots,d_{Q,k-1})
    \right|
    \leq C_MN\tau^k.
\end{equation}

Furthermore,
\[
    \|P_{\geq M+1}\Delta_1\|_\infty
    \leq C_MN\tau^{M+1},
\]
and
\[
    \left|
       \overline{\Delta_1(\mathbf d_Q)}
       -
       \overline{\Delta_1(\mathbf0)}
    \right|
    \leq C_MN(C_M\tau)^N.
\]
All constants are independent of $N$, the generation,
and the parent $Q$.
\end{lemma}

\begin{proof}
Fix a parent $Q$, and write $d_r=d_{Q,r}$.
Throughout the proof, $M$ is fixed, and $N$ is taken
sufficiently large so that $N>2M+2$ and
$C_M\tau<1/2$.

We first recall that
\[
    V_Q(z)-V_Q(w)
    =
    (z-w)\bigl(1+U(z,w)\bigr),
\]
where
$
    U(z,w)=\sum_{r=1}^M d_rS_r(z,w).
$
Since the scalar factor $b_Q$ and the unperturbed
sibling distances contribute only to the constant
Fourier mode, the shape-dependent part of $\Delta_1$
is
\[
    \Delta_1(j;\mathbf d_Q)-\Delta_1(j;\mathbf0)
    =
    \Re\sum_{h\ne j}
    \log\bigl(1+U(\zeta_j,\zeta_h)\bigr).
\]
We will organize this expansion by weighted degree.

\medskip
\noindent
\textit{Step 1. Expansion and weighted-tail estimate.} By \eqref{eq:Sr-definition},
\[
    S_r(z,w)=\sum_{\nu=0}^r z^{r-\nu}w^\nu.
\]
Hence, for $|z|=|w|=1$,
$
    |S_r(z,w)|\leq r+1.$
Using $|d_r|\leq C_M\tau^r$, we obtain
$
    |U(z,w)|\leq C_M\tau.
$ We truncate the logarithmic expansion at order $N$:
\[
    \log(1+U)
    =
    \sum_{\ell=1}^N
    \frac{(-1)^{\ell+1}}{\ell}U^\ell
    +
    R_N(U),
\]
where
$
    |R_N(U)|\leq (C_M\tau)^{N+1}.
$ The index $\ell$ denotes the Taylor order. A term
arising from $U^\ell$ has the form
\begin{equation}\label{eq:D1-general-monomial}
    d_{r_1}\cdots d_{r_\ell}
    S_{r_1}(z,w)\cdots S_{r_\ell}(z,w),
    \qquad 1\leq r_i\leq M.
\end{equation}
Set
$
    m=r_1+\cdots+r_\ell.
$
Thus the  monomial has weighted degree $m$,
and
\[
    |d_{r_1}\cdots d_{r_\ell}|
    \leq C_M^m\tau^m.
\]
Moreover, since each $S_{r_i}$ is homogeneous
of degree $r_i$, their product is homogeneous
of degree $m$ in $(z,w)$. In particular, it can be written as
\[
    S_{r_1}(z,w)\cdots S_{r_\ell}(z,w)
    =
    \sum_{q=0}^m c_q z^{m-q}w^q.
\]

We next estimate the total contribution of all terms
of a fixed weighted degree $m$. This estimate will
be used below to control the higher-order remainders. For fixed $m$, the ordered tuples
$(r_1,\ldots,r_\ell)$ satisfy
\[
    r_1+\cdots+r_\ell=m,
    \qquad 1\leq r_i\leq M.
\]
Summing over all possible Taylor orders $\ell$,
the number of such tuples is at most $2^{m-1}$. Since each term in \eqref{eq:D1-general-monomial}
is bounded by $C_M^m\tau^m$ on the unit circle,
and the factor $1/\ell$ in the logarithmic expansion
is at most $1$, the sum of the absolute values
of all terms of weighted degree $m$ is bounded by
$
    (C_M\tau)^m,
$
after enlarging $C_M$. Consequently, for every $L\geq1$,
the total contribution of all terms of weighted
degree at least $L$ is bounded by
\begin{equation}\label{eq:D1-weighted-tail}
    \sum_{m\geq L}(C_M\tau)^m
    \leq C_M(C_M\tau)^L.
\end{equation}
Here the sum over $m$ is taken over the finite
Taylor expansion. In particular, all possible
Taylor orders $\ell$ have already been accounted
for in \eqref{eq:D1-weighted-tail}.

\medskip
\noindent
\textit{Step 2. Weighted degree and Fourier frequency.} We now explain why a term of weighted degree $m$
produces Fourier frequency $m$, modulo $N$. As observed above, its dependence on $(z,w)$
is a linear combination of
\[
    z^{m-q}w^q,\qquad 0\leq q\leq m.
\]
For the $N$-th roots of unity, we have
\[
    \sum_{h\ne j}\zeta_h^q
    =
    \begin{cases}
       (N-1)\zeta_j^q, & N\mid q,\\
       -\zeta_j^q, & N\nmid q.
    \end{cases}
\]
Therefore,
\[
    \sum_{h\ne j}
    \zeta_j^{m-q}\zeta_h^q
\]
is always a constant multiple of $\zeta_j^m$. 

Thus, after summing over the siblings, every
term of weighted degree $m$ contributes only
to Fourier frequency $m$ modulo $N$. When $m<N$, the identity simplifies to
\[
    \sum_{h\ne j}
    \zeta_j^{m-q}\zeta_h^q
    =
    \begin{cases}
       (N-1)\zeta_j^m, & q=0,\\
       -\zeta_j^m, & 1\leq q\leq m.
    \end{cases}
\]
This identifies the low-frequency contributions
without any reduction modulo $N$.

\medskip
\noindent
\textit{Step 3. The principal contribution to the
$k$-th Fourier mode.} Fix $1\leq k\leq M$. By Step 2, the contribution of weighted degree
$k$ belongs to the $k$-th Fourier mode. The term which is linear in $d_k$ comes from
the linear term $U$ in the logarithmic expansion:
\[
    d_k\sum_{h\ne j}S_k(\zeta_j,\zeta_h).
\]
Using the root-of-unity identities,
\[
\begin{split}
    \sum_{h\ne j}S_k(\zeta_j,\zeta_h)
    &=
    \sum_{\nu=0}^k
    \zeta_j^{k-\nu}
    \sum_{h\ne j}\zeta_h^\nu
\\
    &=
    (N-1)\zeta_j^k-k\zeta_j^k
\\
    &=
    (N-k-1)\zeta_j^k.
\end{split}
\]
Hence the linear contribution is exactly
\[
    (N-k-1)d_k\zeta_j^k.
\]

All other terms of weighted degree $k$ involve
only $d_1,\ldots,d_{k-1}$.
Indeed, any product containing $d_k$ and another
shape parameter has weighted degree greater
than $k$. We denote their combined coefficient by
\[
    \mathcal P_{N,k}(d_1,\ldots,d_{k-1}).
\] We now verify its size and, at the same time,
make the dependence on $N$ explicit. Consider an ordered tuple
$(r_1,\ldots,r_\ell)$ with
$
    r_1+\cdots+r_\ell=k.
$
Write
\[
    S_{r_1}(z,w)\cdots S_{r_\ell}(z,w)
    =
    \sum_{q=0}^k c_qz^{k-q}w^q.
\]
Since each $S_r$ has leading coefficient $1$,
$
    c_0=1.
$
Also, evaluating the product at $z=w=1$ gives
\[
    \sum_{q=0}^k c_q
    =
    \prod_{i=1}^{\ell}(r_i+1).
\]
Thus, after summing over $h\ne j$, the coefficient
of $\zeta_j^k$ is exactly
\[
\begin{split}
    (N-1)c_0-\sum_{q=1}^k c_q
    &=
    N-\prod_{i=1}^{\ell}(r_i+1).
\end{split}
\]
Since $r_1+\cdots+r_\ell=k\leq M$, this
coefficient is bounded by $C_MN$. In particular, the polynomial can be written
explicitly as
\[
\begin{split}
    \mathcal P_{N,k}(d_1,\ldots,d_{k-1})
    =\sum_{\ell=2}^k
    \frac{(-1)^{\ell+1}}{\ell}
    \sum_{\substack{
        r_1+\cdots+r_\ell=k\\
        r_i\geq1
    }}
    \left(
       N-\prod_{i=1}^{\ell}(r_i+1)
    \right)
    d_{r_1}\cdots d_{r_\ell}.
\end{split}
\]
For $k=1$, the sum is empty and
$\mathcal P_{N,1}=0$.

Every monomial in this expression has weighted
degree exactly $k$, and its coefficient depends
only on $N$ and $k$. Furthermore, each parameter product is bounded
by $C_M\tau^k$. Since there are only $O_M(1)$
such terms for $k\leq M$, we obtain
\[
    |\mathcal P_{N,k}(d_1,\ldots,d_{k-1})|
    \leq C_MN\tau^k.
\]
This proves \eqref{eq:PNk-size}. Notice that the factor $N$ comes from the
single sum over the siblings; no additional
factor of $N$ is introduced.

\medskip
\noindent
\textit{Step 4. The remainder in the $k$-th mode.} We have identified all contributions of weighted
degree exactly $k$. We must still account for
higher weighted degrees that may contribute to
the same real Fourier mode after reduction
modulo $N$. Since $\zeta_j^N=1$, frequency $m$ can contribute
to the real $k$-th mode only when
\[
    m\equiv k\pmod N
    \qquad\text{or}\qquad
    m\equiv-k\pmod N.
\]
Because $k\leq M<N/2$, the smallest positive
weighted degree, other than $m=k$, satisfying
this condition is $m=N-k$. Thus every additional contribution to the
$k$-th real mode has weighted degree at least
$N-k$.

By the weighted-tail estimate
\eqref{eq:D1-weighted-tail}, its total size
before summation over siblings is bounded by
\[
    C_M(C_M\tau)^{N-k}.
\]
Summing over the $N-1$ siblings gives
\[
    C_MN(C_M\tau)^{N-k}.
\]
The Fourier projection $P_k$ has uniformly
bounded $L^\infty$ operator norm, so it
introduces no further factor of $N$. Finally, the truncated logarithmic remainder
contributes at most
\[
    C_MN(C_M\tau)^{N+1},
\]
which is absorbed into the preceding bound. We conclude that
\[
\begin{split}
    P_k\Delta_1
    &=
    \Re\left[
       \left(
          (N-k-1)d_k
          +
          \mathcal P_{N,k}(d_1,\ldots,d_{k-1})
       \right)\zeta_j^k
    \right]
\\
    &\qquad+
    O_M\left(N(C_M\tau)^{N-k}\right).
\end{split}
\]

\medskip
\noindent
\textit{Step 5. The higher Fourier modes.} We next estimate $P_{\geq M+1}\Delta_1$. By Step 2, all terms of weighted degree at most
$M$ contribute only to the first $M$ Fourier
modes, since $N>2M$. Therefore,
\[
    P_{\geq M+1}\Delta_1
\]
depends only on terms of weighted degree at
least $M+1$, together with the truncated
Taylor remainder.

By \eqref{eq:D1-weighted-tail}, the total
supremum norm of these terms, after summing
over the siblings, is bounded by
\[
    C_MN\tau^{M+1}.
\]
To see that the projection does not introduce
an additional factor of $N$, recall that
\[
    P_{\geq M+1}
    =
    I-P_0-\sum_{k=1}^M P_k.
\]
Since
\[
    \|P_0\|_{L^\infty\to L^\infty}\leq1,
    \qquad
    \|P_k\|_{L^\infty\to L^\infty}\leq2,
\]
we have
\[
    \|P_{\geq M+1}\|_{L^\infty\to L^\infty}
    \leq 2M+2.
\]
As $M$ is fixed, this factor can be absorbed
into $C_M$. Hence
\[
    \|P_{\geq M+1}\Delta_1\|_\infty
    \leq C_MN\tau^{M+1}.
\]

\medskip
\noindent
\textit{Step 6. Stability of the mean.} Finally, we consider the change in the mean
caused by the shape parameters. By Step 2, every term of weighted degree
$m<N$ contributes a multiple of $\zeta_j^m$.
Its discrete average over $j$ therefore
vanishes. Consequently, a nonzero contribution to
the mean can arise only from weighted
degrees $m\geq N$. Applying \eqref{eq:D1-weighted-tail} with
$L=N$, and then summing over the siblings,
we obtain
\[
    \left|
       \overline{\Delta_1(\mathbf d_Q)}
       -
       \overline{\Delta_1(\mathbf0)}
    \right|
    \leq
    C_MN(C_M\tau)^N.
\]
The Taylor remainder is of still higher order
and is absorbed into the same estimate. This completes the proof.
\end{proof}

\subsection{The exterior field}

\begin{lemma}\label{lem:D2-general}
Assume that
\[
    |d_{Q,r}|\leq C_M\tau^r,
    \qquad 1\leq r\leq M,
\]
and $|\log b_Q|\leq\varepsilon_0$.

Then, for every $1\leq k\leq M$,
\[
    \|P_k\Delta_2\|_\infty
    \leq C_MN\tau^k.
\]
More precisely, there exists a polynomial
$T_{Q,k}$ depending only on
$d_{Q,1},\ldots,d_{Q,k-1}$ such that
\[
\begin{split}
    P_k\Delta_2(j)
    &=
    -\Re\left[
        T_{Q,k}(d_{Q,1},\ldots,d_{Q,k-1})
        \zeta_j^k
    \right]
\\
    &\qquad
    +O_M\left(N(C_M\tau)^{N-k}\right),
\end{split}
\]
where
\[
    \left|
    T_{Q,k}(d_{Q,1},\ldots,d_{Q,k-1})
    \right|
    \leq C_MN\tau^k.
\]

Furthermore,
\[
    \|P_{\geq M+1}\Delta_2\|_\infty
    \leq C_MN\tau^{M+1},
\]
and
\[
    |\overline{\Delta_2}|
    \leq C_MN(C_M\tau)^{N/(M+1)}.
\]
\end{lemma}

\begin{proof}
Fix $Q\in K_{n-1}$, and write $d_r=d_{Q,r}$.
Set
\[
    B_p=\frac{A_p(Q)}{p}(\rho b_Qr_Q)^p.
\]
By Lemma~\ref{lem:Ap} and the bound on $b_Q$,
\[
    |B_p|\leq N(C_M\tau)^p,
    \qquad p\geq1.
\]

Recall that
\[
    \Delta_2(j)
    =-\Re\sum_{p\geq1}B_pV_Q(\zeta_j)^p,
    \qquad
    V_Q(z)=z\left(1+\sum_{r=1}^M d_rz^r\right).
\]
We first estimate the contributions of each Fourier
frequency. This will give the estimates for the low
modes, the higher modes, and the mean.

\medskip
\noindent
\textit{Fourier coefficients.}
Expanding in powers of $z$, write
\[
    \sum_{p\geq1}B_pV_Q(z)^p
    =
    \sum_{m\geq1}T_{Q,m}z^m.
\]
We claim that
\[
    |T_{Q,m}|\leq N(C_M\tau)^m,
    \qquad m\geq1.
\]

Indeed, a term of degree $m$ arising from
$B_pV_Q(z)^p$ must satisfy
\[
    m=p+r_1+\cdots+r_\ell,
    \qquad 1\leq r_i\leq M.
\]
The corresponding parameter product satisfies
\[
    |d_{r_1}\cdots d_{r_\ell}|
    \leq C_M^{m-p}\tau^{m-p}.
\]
Since there are at most $(M+1)^p$ terms in the
expansion of $V_Q(z)^p$, the total contribution
of degree $m$ from this Taylor order is bounded by
\[
    N(C_M\tau)^m,
\]
after enlarging $C_M$. Moreover, only $1\leq p\leq m$ can contribute
to degree $m$. Summing over these Taylor orders
and absorbing the factor $m\leq2^m$ into $C_M$,
we obtain the claimed estimate.

Consequently, for any $L\geq1$,
\[
    \sum_{m\geq L}|T_{Q,m}|
    \leq
    C_MN(C_M\tau)^L,
\]
provided $C_M\tau$ is sufficiently small.
In particular, the expansion converges absolutely
on the unit circle.

\medskip
\noindent
\textit{The first $M$ Fourier modes.}
Fix $1\leq k\leq M$. Since $V_Q(z)^p$ has lowest degree $p$,
only $1\leq p\leq k$ can contribute directly
to the coefficient $T_{Q,k}$. Moreover, a term of degree $k$ arising from
$V_Q(z)^p$ contains shape parameters of total
weight $k-p\leq k-1$. Thus $T_{Q,k}$ is a polynomial in
$d_1,\ldots,d_{k-1}$, with
\[
    |T_{Q,k}|\leq C_MN\tau^k.
\]
For $k=1$, this polynomial is independent
of the shape parameters.

Since $\zeta_j^N=1$, a higher frequency $m$
can contribute to the same real Fourier mode
only if
\[
    m\equiv k\pmod N
    \qquad\text{or}\qquad
    m\equiv-k\pmod N.
\]
Apart from $m=k$, the smallest possible
positive frequency is $N-k$. The total contribution of these higher
frequencies is therefore bounded by
\[
    \sum_{m\geq N-k}|T_{Q,m}|
    \leq C_MN(C_M\tau)^{N-k}.
\]
It follows that
\[
    P_k\Delta_2(j)
    =
    -\Re\left[
        T_{Q,k}(d_1,\ldots,d_{k-1})
        \zeta_j^k
    \right]
    +
    O_M\left(N(C_M\tau)^{N-k}\right).
\]
In particular,
\[
    \|P_k\Delta_2\|_\infty
    \leq C_MN\tau^k.
\]

\medskip
\noindent
\textit{Higher modes and the mean.}
Every term of degree at most $M$ contributes
only to the first $M$ Fourier modes.
Hence
\[
\begin{split}
    \|P_{\geq M+1}\Delta_2\|_\infty
    &\leq
    \sum_{m\geq M+1}|T_{Q,m}|\\
    &\leq C_MN\tau^{M+1}.
\end{split}
\]
Here we have also included frequencies which
reduce modulo $N$ to lower modes; this only
enlarges the upper bound. Finally, the discrete average of $\zeta_j^m$
vanishes unless $N$ divides $m$.
Thus
\[
\begin{split}
    |\overline{\Delta_2}|
    &\leq
    \sum_{m\geq N}|T_{Q,m}|\\
    &\leq C_MN(C_M\tau)^N\\
    &\leq C_MN(C_M\tau)^{N/(M+1)}.
\end{split}
\]
This completes the proof.
\end{proof}

\subsection{The parent replacement term}

\begin{lemma}\label{lem:D3-general}
Fix $n\ge 1$. Assume that, for every vertex $Q$
of generation at most $n-1$,
\[
    |d_{Q,r}|\le C_M\tau^r,
    \qquad 1\le r\le M,
    \qquad\text{and}\qquad
    |\log b_Q|\le\varepsilon_0.
\]
Then, provided $N$ is sufficiently large
depending on $M$ and the fixed $s<1$,
\[
    \|\Delta_3\|_\infty
    \le C_M N(C_M\tau)^{N/(M+1)},
\]
where the supremum is taken over $K_n$.

In particular, for fixed $M$ and $s<1$,
\[
    \|\Delta_3\|_\infty
    =o_M(N\tau^{M+1})
    \qquad\text{as }N\to\infty.
\]
\end{lemma}

\begin{proof}
Fix $x=x_{Qj}\in K_n$, where $Q\in K_{n-1}$ is its parent.
For each $y\in K_{n-1}\setminus\{Q\}$, denote the displacements
of the children of $y$ by
\[
 h_{y,k}=x_{yk}-y
 =\rho b_y r_y V_y(\zeta_k),
 \qquad
 V_y(z)=z+\sum_{r=1}^M d_{y,r}z^{r+1}.
\]
Since all the points $y$ belong to generation $n-1$, we have
$r_y=r_Q$. Writing
\[
 D_y(x):=\sum_{k=0}^{N-1}
 \bigl(\log|x-y-h_{y,k}|-\log|x-y|\bigr),
\]
we have
\[
 \Delta_3(x)=\sum_{y\in K_{n-1}\setminus\{Q\}}D_y(x).
\]

\smallskip
\noindent\emph{Step 1: cancellation of the low-order moments.}
For every integer $q\ge1$,
\[
 V_y(z)^q
 =z^q\left(1+\sum_{r=1}^M d_{y,r}z^r\right)^q.
\]
Hence every monomial in this polynomial has degree between
$q$ and $(M+1)q$. If $(M+1)q<N$, none of these degrees is a
multiple of $N$. The root-of-unity identities therefore give
\[
 \sum_{k=0}^{N-1}V_y(\zeta_k)^q=0,
 \qquad (M+1)q<N,
\]
and consequently
\begin{equation}\label{eq:D3-general-moments}
 \sum_{k=0}^{N-1}h_{y,k}^q=0,
 \qquad 1\le q<L_N,
 \qquad L_N:=\left\lceil\frac{N}{M+1}\right\rceil.
\end{equation}

\smallskip
\noindent\emph{Step 2: the contribution of one exterior parent.}
Set $w=x-y$ and
\[
 \eta_y=\max_{0\le k<N}\frac{|h_{y,k}|}{|w|}.
\]
We shall verify below that $\eta_y\le1/2$. The absolutely
convergent logarithmic expansion then yields
\[
 \begin{aligned}
 D_y(x)
 &=\sum_{k=0}^{N-1}\Re\log\left(1-\frac{h_{y,k}}{w}\right)=-\Re\sum_{q=1}^{\infty}\frac{1}{qw^q}
       \sum_{k=0}^{N-1}h_{y,k}^q=-\Re\sum_{q=L_N}^{\infty}\frac{1}{qw^q}
       \sum_{k=0}^{N-1}h_{y,k}^q,
 \end{aligned}
\]
where the last equality follows from
\eqref{eq:D3-general-moments}. Thus
\[
 |D_y(x)|
 \le N\sum_{q=L_N}^{\infty}\frac{\eta_y^q}{q}
 \le\frac{N}{L_N}\frac{\eta_y^{L_N}}{1-\eta_y}
 \le C_M\eta_y^{L_N},
\]
because $N/L_N\le M+1$.

\smallskip
\noindent\emph{Step 3: the genealogical coronas.}
For $\ell\ge1$, let $\mathcal Y_\ell(Q)$ be the set of parents
$y\in K_{n-1}$ whose branches first separate from that of $Q$
$\ell$ generations above $Q$. As in the preceding geometric
estimates, put
\[
 R_\ell=\lambda^{-\ell}r_Q,
 \qquad
 \#\mathcal Y_\ell(Q)=(N-1)N^{\ell-1}\le N^\ell.
\]
The separation of distinct sibling branches, together with the
fact that their descendants are contained in balls of radius
$O_M(\lambda R_\ell)$ around the corresponding centers, gives
\[
 |x-y|\ge c_M\frac{R_\ell}{N},
 \qquad y\in\mathcal Y_\ell(Q).
\]
Here $c_M>0$ is independent of $N$, the generation, $Q$ and
$\ell$, provided $\tau=N\lambda$ is sufficiently small.
On the other hand, since $|\log b_y|\le\varepsilon_0$ and
$|d_{y,r}|\le C_M\tau^r$,
\[
 |h_{y,k}|\le C_Mr_Q=C_M\lambda^\ell R_\ell.
\]
It follows that
\begin{equation}\label{eq:D3-general-eta}
 \eta_y\le C_M N\lambda^\ell
 =C_M\tau\lambda^{\ell-1},
 \qquad y\in\mathcal Y_\ell(Q).
\end{equation}
Taking $N$ large enough makes the right-hand side at most
$1/2$ for every $\ell\ge1$, as required in Step 2.

Summing the estimate for $D_y(x)$ over the coronas, we obtain
\[
 \begin{aligned}
 |\Delta_3(x)|
 &\le C_M\sum_{\ell\ge1}
       N^\ell\bigl(C_M\tau\lambda^{\ell-1}\bigr)^{L_N}=C_M N(C_M\tau)^{L_N}
       \sum_{\ell\ge1}(N\lambda^{L_N})^{\ell-1}.
 \end{aligned}
\]
Since $L_N\ge N/(M+1)$, and
$\lambda=N^{-1/s}$ with $s<1$, we have
$N\lambda^{L_N}\le1/2$ for sufficiently large $N$.
Consequently,
\[
 |\Delta_3(x)|
 \le C_MN(C_M\tau)^{L_N}
 \le C_MN(C_M\tau)^{N/(M+1)}.
\]
The estimate is uniform in $x$ and in the generation,
which proves the first assertion.

Finally, for $M$ and $s<1$ fixed, we have
$\tau=N^{1-1/s}\to0$ as $N\to\infty$. Since
$N/(M+1)\to\infty$,
\[
 \frac{(C_M\tau)^{N/(M+1)}}{\tau^{M+1}}\longrightarrow0.
\]
This proves the last assertion and completes the proof.
\end{proof}

\subsection{Finite-order cancellation}

We now choose the shape parameters
$d_{Q,1},\ldots,d_{Q,M}$, keeping the scalar
parameters fixed. By Lemmas~\ref{lem:D1-general}
and~\ref{lem:D2-general}, the principal part
of the $k$-th Fourier mode depends linearly
on $d_{Q,k}$ and otherwise only on
$d_{Q,1},\ldots,d_{Q,k-1}$.
This allows us to choose the shape parameters
successively.
\begin{lemma}[One-pass triangular correction]
\label{lem:one-pass}
Fix $M\geq1$, $0<s<1$, and $n\geq1$.
Suppose that $K_{n-1}$ has been constructed and
that, for every vertex $P$ of generation at most $n-2$,
\[
    |d_{P,r}|\leq C_M\tau^r,
    \qquad 1\leq r\leq M,
    \qquad
    |\log b_P|\leq\varepsilon_0.
\]
For every $Q\in K_{n-1}$, fix a scalar parameter
$b_Q$ satisfying
\[
    |\log b_Q|\leq\varepsilon_0.
\]

Then, for $N$ sufficiently large depending on
$M$ and $s$, the shape parameters can be chosen
successively, in the order $k=1,\ldots,M$, so that
\[
    |d_{Q,k}|\leq C_M\tau^k,
    \qquad 1\leq k\leq M,
\]
and
\begin{equation}\label{eq:general-total-osc}
    \|F_Q-\overline F_Q\|_\infty
    \leq C_MN\tau^{M+1},
    \qquad Q\in K_{n-1}.
\end{equation}

Moreover, keeping all scalar parameters fixed,
the change in the mean satisfies
\begin{equation}\label{eq:general-mean-shape}
    \left|
        \overline{F_Q(\mathbf d,\mathbf t)}
        -
        \overline{F_Q(\mathbf 0,\mathbf t)}
    \right|
    \leq C_MN(C_M\tau)^{N/(M+1)}.
\end{equation}
Here $\mathbf 0$ means that all shape parameters
of generation $n-1$ are set to zero.
The constants are independent of the generation
and of $Q$.
\end{lemma}

\begin{proof}
We choose the shape parameters at each parent
$Q\in K_{n-1}$, keeping all scalar parameters fixed.

\medskip
\noindent
\textit{Step 1. Choosing the shape parameters.} Fix $Q\in K_{n-1}$.
By Lemmas~\ref{lem:D1-general} and
\ref{lem:D2-general}, the contribution to the
$k$-th real Fourier mode of $\Delta_1+\Delta_2$,
before accounting for the exponentially small
aliasing error, is
\[
\begin{split}
    \Re\Big[
    \big(
        &(N-k-1)d_{Q,k}\\
        &+\mathcal P_{N,k}
          (d_{Q,1},\ldots,d_{Q,k-1})\\
        &-T_{Q,k}
          (d_{Q,1},\ldots,d_{Q,k-1})
    \big)\zeta_j^k
    \Big].
\end{split}
\]
Both $\mathcal P_{N,k}$ and $T_{Q,k}$ depend
only on the previously chosen parameters
$d_{Q,1},\ldots,d_{Q,k-1}$.
We therefore define
\[
    d_{Q,k}
    =
    \frac{
        T_{Q,k}(d_{Q,1},\ldots,d_{Q,k-1})
        -
        \mathcal P_{N,k}
          (d_{Q,1},\ldots,d_{Q,k-1})
    }{N-k-1}.
\]
This cancels the displayed coefficient exactly.

We verify the admissibility bounds by induction
on $k$. For $k=1$, since $\mathcal P_{N,1}=0$ and
$T_{Q,1}$ is independent of the shape parameters,
\[
    d_{Q,1}=\frac{T_{Q,1}}{N-2}.
\]
The estimate $|T_{Q,1}|\leq C_MN\tau$ gives
\[
    |d_{Q,1}|\leq C_M\tau.
\]
Suppose now that
\[
    |d_{Q,r}|\leq C_{M,r}\tau^r,
    \qquad 1\leq r<k,
\]
with constants independent of $N$, $n$, and $Q$. The weighted-degree estimates in
Lemmas~\ref{lem:D1-general} and
\ref{lem:D2-general} give
\[
\begin{split}
    &\left|
        \mathcal P_{N,k}
        (d_{Q,1},\ldots,d_{Q,k-1})
    \right|\\
    &\qquad+
    \left|
        T_{Q,k}
        (d_{Q,1},\ldots,d_{Q,k-1})
    \right|
    \leq C_{M,k}N\tau^k.
\end{split}
\]
These estimates involve only the already
determined coefficients. Since $N-k-1\geq N/2$ for $N\geq2M+4$,
the recursive definition yields
\[
    |d_{Q,k}|\leq C_{M,k}\tau^k.
\]
As $M$ is fixed, the induction gives
\[
    |d_{Q,k}|\leq C_M\tau^k,
    \qquad 1\leq k\leq M,
\]
with a uniform constant. 

We perform this construction for every
$Q\in K_{n-1}$.

\medskip
\noindent
\textit{Step 2. Controlling the remaining oscillation.} By construction, the coefficients displayed
in Step 1 vanish for every $1\leq k\leq M$. Now that all shape parameters satisfy the
admissibility bounds, the remainder estimates
of Lemmas~\ref{lem:D1-general} and
\ref{lem:D2-general} apply. They give
\[
    \|P_k(\Delta_1+\Delta_2)\|_\infty
    \leq C_MN(C_M\tau)^{N-k},
    \qquad 1\leq k\leq M.
\]
For sufficiently large $N$, these bounds are
at most $C_MN\tau^{M+1}$.

The same lemmas give
\[
    \|P_{\geq M+1}(\Delta_1+\Delta_2)\|_\infty
    \leq C_MN\tau^{M+1}.
\]
Since $M$ is fixed, combining the low and
higher modes yields
\[
    \left\|
        P_{\mathrm{nc}}(\Delta_1+\Delta_2)
    \right\|_\infty
    \leq C_MN\tau^{M+1}.
\]

It remains to include $\Delta_3$.
This term depends on the new parameters at
all parents of generation $n-1$.
Since they have all been chosen and satisfy
the admissibility bounds,
Lemma~\ref{lem:D3-general} applies:
\[
    \|\Delta_3\|_\infty
    \leq C_MN(C_M\tau)^{N/(M+1)}
    =o_M(N\tau^{M+1}).
\]
Consequently,
\[
    \|F_Q-\overline F_Q\|_\infty
    \leq C_MN\tau^{M+1},
\]
uniformly in $Q$. This proves
\eqref{eq:general-total-osc}.

\medskip
\noindent
\textit{Step 3. Stability of the mean.} We compare the constructed configuration
with the one in which all new shape parameters
are zero, keeping the scalar parameters fixed. By Lemma~\ref{lem:D1-general},
\[
    \left|
        \overline{\Delta_1(\mathbf d)}
        -
        \overline{\Delta_1(\mathbf 0)}
    \right|
    \leq C_MN(C_M\tau)^N.
\] 
By Lemma~\ref{lem:D2-general}, the means of
$\Delta_2$ in both configurations are bounded by
\[
    C_MN(C_M\tau)^{N/(M+1)}.
\]
Lemma~\ref{lem:D3-general} gives the same bound
for $\Delta_3$ in both configurations. Combining these estimates gives
\[
    \left|
        \overline{F_Q(\mathbf d,\mathbf t)}
        -
        \overline{F_Q(\mathbf 0,\mathbf t)}
    \right|
    \leq C_MN(C_M\tau)^{N/(M+1)}.
\]
This proves \eqref{eq:general-mean-shape}
and completes the proof.
\end{proof}

\subsection{Induction in the generations}

Suppose that
\[
    |g_{n-1}(Q)-c_{n-1}|
    \leq A N^{-(n-1)},
    \qquad Q\in K_{n-1},
    \qquad A=K_MN\tau^{M+1},
\]
and that the parameters of the preceding generations
satisfy the admissibility conditions.

We shall choose the new parameters so that the same
estimate holds at generation $n$.

\medskip
\noindent
\textit{Choosing the scalar parameters.} We first set all new shape parameters to zero and
consider the reference configuration with
$b_Q=1$ for every $Q\in K_{n-1}$.
As in Section~\ref{sec:first-mode-induction},
we choose
\[
    c_n=c_{n-1}+\gamma_n,
    \qquad
    \gamma_n=N^{-n}
    \left[
        (N-1)\log(\rho\lambda^{n-1})+\log N
    \right].
\]

Define the reference mean
\[
    H_Q
    :=
    g_{n-1}(Q)
    +N^{-n}\overline F_Q(\mathbf0,\mathbf0).
\]
The local term in this configuration is independent
of $Q$ and equals $N^n\gamma_n$.
By Lemmas~\ref{lem:D2-general} and
\ref{lem:D3-general},
\[
    |H_Q-c_n|
    \leq
    A N^{-(n-1)}
    +C_MN^{-n}N(C_M\tau)^{N/(M+1)}.
\]
For $N$ sufficiently large, the second term is
smaller than $A N^{-(n-1)}$. Hence
\[
    |H_Q-c_n|
    \leq C_MA N^{-(n-1)}.
\]

We now define
\begin{equation}\label{eq:general-tQ}
    t_Q
    =
    \frac{N^n(c_n-H_Q)}{N-1},
    \qquad
    b_Q=e^{t_Q}.
\end{equation}
It follows that
\[
    |t_Q|\leq C_MA.
\]
In particular, $|t_Q|\leq\varepsilon_0$
provided $N\tau^{M+1}$ is sufficiently small,
as verified in the next subsection.

\medskip
\noindent
\textit{Correcting the mean.} Keep the new shape parameters equal to zero.
Changing $b_Q$ from $1$ to $e^{t_Q}$
increases the local term $\Delta_1$ by exactly
$(N-1)t_Q$. By our choice of $t_Q$,
\[
    H_Q-c_n+N^{-n}(N-1)t_Q=0.
\]
Therefore,
\[
\begin{split}
    &g_{n-1}(Q)
    +N^{-n}\overline F_Q(\mathbf0,\mathbf t)-c_n\\
    &\quad=
    N^{-n}\Big[
        \overline{\Delta_{2,Q}}(\mathbf0,\mathbf t)
        -\overline{\Delta_{2,Q}}(\mathbf0,\mathbf0)\\
    &\qquad\qquad+
        \overline{\Delta_{3,Q}}(\mathbf0,\mathbf t)
        -\overline{\Delta_{3,Q}}(\mathbf0,\mathbf0)
    \Big].
\end{split}
\]

Lemma~\ref{lem:D2-general} gives
\[
    |\overline{\Delta_{2,Q}}(\mathbf0,\mathbf t)|
    +
    |\overline{\Delta_{2,Q}}(\mathbf0,\mathbf0)|
    \leq C_MN(C_M\tau)^N.
\]
By Lemma~\ref{lem:D3-general},
\[
    |\overline{\Delta_{3,Q}}(\mathbf0,\mathbf t)|
    +
    |\overline{\Delta_{3,Q}}(\mathbf0,\mathbf0)|
    \leq C_MN(C_M\tau)^{N/(M+1)}.
\]
Consequently,
\begin{equation}\label{eq:general-mean-after-b}
    \left|
        g_{n-1}(Q)
        +N^{-n}\overline F_Q(\mathbf0,\mathbf t)
        -c_n
    \right|
    \leq
    C_MN^{-n}N(C_M\tau)^{N/(M+1)}.
\end{equation}

\medskip
\noindent
\textit{Closing the induction.} We now keep the scalar parameters fixed and
choose the shape parameters by
Lemma~\ref{lem:one-pass}. The mean estimate \eqref{eq:general-mean-shape},
combined with \eqref{eq:general-mean-after-b},
gives
\[
    \left|
        g_{n-1}(Q)
        +N^{-n}\overline F_Q(\mathbf d,\mathbf t)
        -c_n
    \right|
    \leq
    C_MN^{-n}N(C_M\tau)^{N/(M+1)}.
\]

On the other hand, by
\eqref{eq:general-total-osc},
\[
    \|F_Q-\overline F_Q\|_\infty
    \leq C_MN\tau^{M+1}.
\]
For $x=x_{Qj}\in K_n$, recall that
\[
    g_n(x)
    =
    g_{n-1}(Q)+N^{-n}F_Q(j).
\]
It follows that
\[
\begin{split}
    |g_n(x)-c_n|
    &\leq
    N^{-n}|F_Q(j)-\overline F_Q|\\
    &\quad+
    \left|
        g_{n-1}(Q)+N^{-n}\overline F_Q-c_n
    \right|\\
    &\leq
    C_MN^{-n}N
    \left[
        \tau^{M+1}
        +(C_M\tau)^{N/(M+1)}
    \right]\\
    &\leq
    C_MN^{-n}N\tau^{M+1}.
\end{split}
\]
Choosing $K_M$ sufficiently large, we obtain
\[
    |g_n(x)-c_n|\leq AN^{-n},
    \qquad x\in K_n,
\]
which closes the induction.
\subsection{Choice of the parameters}

\begin{lemma}\label{lem:general-parameter-choice}
Fix $M\ge1$ and suppose
\[
 s<\frac{M+1}{M+2}.
\]
With
\[
 \lambda=N^{-1/s},
 \qquad
 \tau=N\lambda=N^{1-1/s},
\]
one has
\[
 \tau\longrightarrow0,
 \qquad
 N\tau^{M+1}\longrightarrow0
\]
as $N\to\infty$.
\end{lemma}

\begin{proof}
Since $s<1$, $\tau=N^{1-1/s}\to0$. Moreover,
\[
 N\tau^{M+1}
 =
 N^{M+2-(M+1)/s},
\]
and the exponent is negative exactly when
\[
 s<\frac{M+1}{M+2}.
\]
\end{proof}
Thus, for every fixed $M$ and
$s<(M+1)/(M+2)$, we may first choose $K_M$ and then take $N$
sufficiently large so that all the smallness and admissibility
conditions above hold. Iterating the construction gives
\[
 \sup_{x\in K_n}|g_n(x)-c_n|
 \le A N^{-n}.
\]
The argument of Section~\ref{sec:first-mode-induction} also shows
that $(c_n)$ is a Cauchy sequence.

\subsection{Passage to the limiting equilibrium measure}

Let
\[
 \mu_n=N^{-n}\sum_{x\in K_n}\delta_x.
\]
The measures $\mu_n$ converge weakly to the natural Bernoulli measure
$\mu$ on the limiting compact set $K$.

The geometric argument of Section~\ref{sec:geometry} is unchanged by the polynomial
corrections. Indeed, for fixed $M$,
\[
 V_Q(\zeta)=\zeta+\sum_{k=1}^M d_{Q,k}\zeta^{k+1},
 \qquad |d_{Q,k}|\le C_M\tau^k,
\]
is a uniformly small perturbation of the identity. Hence the cylinder
diameter and separation estimates of Section~\ref{sec:geometry} remain valid. Since
every generation-$n$ cylinder has mass
\[
 N^{-n}=\lambda^{sn},
\]
the standard cylinder-counting argument gives
\[
 \mu(B(x,r))\simeq r^s,
 \qquad x\in K,\quad 0<r<\operatorname{diam}K.
\]
Thus $K$ and $\mu$ are $s$-Ahlfors regular.

\begin{lemma}\label{lem:general-limit-potential}
The logarithmic potential
\[
 G(z)=\int_K\log|z-y|\,d\mu(y)
\]
is continuous and
\[
 G(x)=c_\infty
 \qquad(x\in K),
\]
where $c_\infty=\lim_n c_n$. Consequently $\mu$ is the logarithmic
equilibrium measure of $K$.
\end{lemma}

\begin{proof}
Since $\mu$ is $s$-Ahlfors regular and $s>0$, the logarithmic
singularity is uniformly integrable. In particular $G$ is continuous.

Fix $x\in K$, and let $x_n\in K_n$ be the center associated with the
generation-$n$ cylinder containing $x$. Then $x_n\to x$. The weak
convergence $\mu_n\rightharpoonup\mu$, together with uniform
integrability near the diagonal, gives
\[
 G(x)=\lim_{n\to\infty}g_n(x_n).
\]
Since
\[
 |g_n(x_n)-c_n|\le AN^{-n}
\]
and $c_n\to c_\infty$, it follows that $G(x)=c_\infty$ on $K$.

Thus $G-c_\infty$ is continuous in $\mathbb C$, harmonic in
$\mathbb C\setminus K$, vanishes on $K$, and satisfies
\[
 G(z)-c_\infty=\log|z|+O(1)
 \qquad(|z|\to\infty).
\]
It is therefore the Green function of $\mathbb C\setminus K$ with
pole at infinity, and its Riesz measure is $\mu$. Hence $\mu$ is the
equilibrium measure of $K$.
The passage from the discrete potentials to the limiting logarithmic potential follows by the same argument as in \cite[Section 3]{DJJ23}, using the Ahlfors regularity of \(\mu\) and the separation estimates for the cylinders.
\end{proof}

\begin{proof}[Proof of Theorem~\ref{thm:main}]
Fix $0<s<1$ and choose $M$ so large that
\[
 s<\frac{M+1}{M+2}.
\]
Then choose $N$ sufficiently large. By
Lemma~\ref{lem:general-parameter-choice}, the construction above can
be iterated through all generations.

The preceding discussion shows that the limiting Bernoulli measure
$\mu$ is $s$-Ahlfors regular, and
Lemma~\ref{lem:general-limit-potential} identifies it with the
logarithmic equilibrium measure of $K$. Since $s<1$, the compact set
$K$ does not separate the plane, so $\Omega=\mathbb C\setminus K$ is
connected. Therefore the equilibrium measure is harmonic measure
with pole at infinity:
\[
 \omega^\infty=\mu\approx\mathcal H^s|_K.
\]
\end{proof}

\end{document}